\documentclass{amsart}

\usepackage{amsmath}
\usepackage{amssymb}
\usepackage{commath}
\usepackage{mathtools}
\usepackage[all]{xy}

\usepackage{hyperref}

\theoremstyle{plain}
\newtheorem{theorem}{Theorem}
\newtheorem{lemma}[theorem]{Lemma}
\newtheorem{proposition}[theorem]{Proposition}
\newtheorem{corollary}[theorem]{Corollary}
\theoremstyle{definition}
\newtheorem{definition}[theorem]{Definition}
\newtheorem{remark}[theorem]{Remark}

\numberwithin{theorem}{section}
\numberwithin{equation}{section}

\newcommand{\B}{\mathbb{B}}

\newcommand{\C}{\mathbb{C}}
\newcommand{\N}{\mathbb{N}}
\newcommand{\R}{\mathbb{R}}
\newcommand{\T}{\mathbb{T}}
\newcommand{\Z}{\mathbb{Z}}

\newcommand{\U}{\mathrm{U}}

\newcommand{\tr}{\mathrm{tr}}

\newcommand{\cT}{\mathcal{T}}
\newcommand{\cA}{\mathcal{A}}
\newcommand{\cP}{\mathcal{P}}
\newcommand{\cS}{\mathcal{S}}

\newcommand{\End}{\mathrm{End}}

\newcommand{\kb}{{\boldsymbol{k}}}
\newcommand{\kapb}{{\boldsymbol{\kappa}}}
\newcommand{\mbS}{\mathbb{S}}

\begin{document}

%
%
%
%
%
%
%
%
%


\title[Toeplitz Operators, $\mathbb{T}^m$-invariance and Quasi-homogeneous Symbols]{Toeplitz Operators, $\mathbb{T}^m$-invariance \\ and Quasi-homogeneous Symbols}

\author[R. Quiroga-Barranco]{Ra\'ul Quiroga-Barranco}


\address{Centro de Investigaci\'on en Matem\'aticas\\
Guanajuato, Guanajuato\\
M\'exico}

\email{quiroga@cimat.mx}





\begin{abstract}
For a partition $\boldsymbol{k} = (k_1, \dots, k_m)$ of $n$ consider the group $\mathrm{U}(\boldsymbol{k}) = \mathrm{U}(k_1) \times \dots \times \mathrm{U}(k_m)$ block diagonally embedded in $\mathrm{U}(n)$ and the center $\mathbb{T}^m$ of $\mathrm{U}(\boldsymbol{k})$. We study the Toeplitz operators with $\mathbb{T}^m$-invariant symbols acting on the weighted Bergman spaces on the unit ball $\mathbb{B}^n$. We introduce the $(\boldsymbol{k},j)$-quasi-radial quasi-homogeneous symbols as those that are invariant under the group $\mathrm{U}(\boldsymbol{k},j,\mathbb{T})$ obtained from $\mathrm{U}(\boldsymbol{k})$ by replacing the factor $\mathrm{U}(k_j)$ with its center $\mathbb{T}$. These symbols are used to build commutative Banach non-$C^*$ algebras generated by Toeplitz operators. These algebras generalize those from the literature and show that they can be built using groups. We describe the action of such Toeplitz operators on monomials through explicit integral formulas involving the symbols. We prove that every Toeplitz operator with $\mathbb{T}^m$-invariant symbol has an associated Toeplitz operator with $\mathrm{U}(\boldsymbol{k})$-invariant symbol in terms of which we can describe some properties.
\end{abstract}

\maketitle
\section{Introduction}
\subsection{Commutative $C^*$-algebras and beyond}
\label{subsec:commutative}
Toeplitz operators acting on Berg\-man spaces on the unit ball $\B^n$ in $\C^n$ have been an important object of study in operator theory. Among these operators, the so-called radial Toeplitz operators appear as a fundamental example since their description in \cite{KorenblumZhu1995} for the case of the unit disk. Such operators have two possible generalizations to the $n$-dimensional case, which are the radial and the separately radial (also known as quasi-elliptic) Toeplitz operators. Several authors have studied intensively this sort of operators. Among those references studying either radial or separately radial operators one can find \cite{AppuhamyLe2016}, \cite{BHVRadial}, \cite{BVQuasiEllipticI}, \cite{EsmeralMaximenko2016}, \cite{GKVRadial}, \cite{GMVRadial}, \cite{Le2017}, \cite{QRadial}, \cite{SuarezRadial2008}, \cite{ZorbRadial2002}. An interesting feature of this sort of operators is that they commute with each other. This fact can be considered as a source of motivation for the study of the commutativity of Toeplitz operators as found, for example, in \cite{AppuhamyLe2016}, \cite{AxlerCuckovicRao2000}, \cite{BauerChoeKoo2015}, \cite{BVquasir-quasih}, \cite{BVQuasiEllipticI}, \cite{ChoeKooLee2004}, \cite{CuckovicLouhichi2008}, \cite{DOQJFA}, \cite{Le2017}. This naturally leads to consider commutative $C^*$-algebras generated by Toeplitz operators. This is one of the reasons for the importance of proving the existence of commutative $C^*$-algebras generated by Toeplitz operators. Another one, is the fact that they provide concrete examples of commutative algebras (see \cite{BHVRadial} and \cite{GMVRadial}).

Radial and separately radial symbols can be defined through the invariance of suitable groups. When considering special symbols defined by invariance, and for our setup, one starts with a subgroup $H$ of the automorphism group of $\B^n$. Then, for every $\lambda > -1$, the set $L^\infty(\B^n)^H$ (essentially bounded $H$-invariant symbols) yields Toeplitz operators acting on the weighted Bergman space $\cA^2_\lambda(\B^n)$ that generate a $C^*$-algebra that we will denote by $\cT^{(\lambda)}(L^\infty(\B^n)^H)$. It turns out that for a large list of groups $H$, the $C^*$-algebra $\cT^{(\lambda)}(L^\infty(\B^n)^H)$ is commutative. This technique started with \cite{QVUnitBall1} and \cite{QVUnitBall2}, and was later refined in \cite{DOQJFA} for $\B^n$ as well as for more general domains. Two particularly interesting cases appear in this setup. For the unitary group $\U(n)$, the symbols in $L^\infty(\B^n)^{\U(n)}$ are the so-called radial symbols. For the subgroup $\T^n \subset \U(n)$ of diagonal matrices, the symbols in $L^\infty(\B^n)^{\T^n}$ are known as separately radial or quasi-elliptic symbols. These yield the precise definitions of the symbols mentioned in the previous paragraph. In both cases, these families of symbols yield Toeplitz operators generating commutative $C^*$-algebras on every weighted Bergman space over $\B^n$. We refer to \cite{GKVRadial} and \cite{QVUnitBall1}, respectively, for the corresponding results. The commutativity of such $C^*$-algebras, radial and separately radial cases, can be proved using representation theory. This has been done in \cite{QRadial} by computing the isotypic decompositions (see Section~\ref{sec:isotypic} for the relevant definitions) for the actions on $\cA^2_{\lambda}(\B^n)$ of the corresponding groups. Besides such computation, one of the main tools used to prove the commutativity is Schur's Lemma. This fundamental representation theoretic result can be naturally applied mostly because the terms of the isotypic decompositions are irreducible (multiplicity-free decomposition) for these cases.

The next step in this sort of development is to consider algebras generated by Toeplitz operators of a more general sort in some chosen sense. This can be done in a variety of ways, as the references cited above show. In this work, we are mostly interested in considering algebras more general than commutative $C^*$-algebras using the methods of groups and their representation theory. On one hand, we will consider the property of invariance with respect to certain subgroups to obtain special symbols that yield certain Toeplitz operators. On the other hand, the Toeplitz operators that we will consider will not all commute with each other, but a careful choice of some of them will yield commutative Banach algebras. Hence, we will consider at the same time the commuting and the noncommuting case for Toeplitz operators, with the former generating commutative algebras that are only Banach. More precisely, in this work we will consider some particular subgroups $H$ of $\U(n)$ for which the algebra $\cT^{(\lambda)}(L^\infty(\B^n)^H)$ is not necessarily commutative, and we want to still be able to determine non-trivial properties of the Toeplitz operators with symbols in $L^\infty(\B^n)^H$. It turns out that for such an $H$, the isotypic decomposition no longer has irreducible terms (higher multiplicity decomposition). In particular, Schur's Lemma is not enough by itself to conclude the commutativity of suitable Toeplitz operators. Similarly, some tools beyond Schur's Lemma are needed to obtain non-trivial properties of individual Toeplitz operators.

\subsection{Quasi-homogeneous symbols and Toeplitz operators}
\label{subsec:quasi_and_Toeplitz}
There are some interesting sets of symbols that somehow lie in between the radial and the separately radial ones: the quasi-radial symbols. These are given as symbols invariant under the action of the subgroup $\U(k_1) \times \dots \times \U(k_m)$ block diagonally embedded in $\U(n)$, where $\kb = (k_1, \dots, k_m)$ is a partition of $n$. We denote such subgroup by $\U(\kb)$ (see Sections~\ref{sec:isotypic} and \ref{sec:Tm_and_quasi-radial}). This subgroup yields commutative $C^*$-algebras as well (see \cite{VasQuasiRadial}).

A natural step to consider is to broaden the class of algebras generated by Toeplitz operators with special symbols. This was the approach used in \cite{VasQuasiRadial}, \cite{BVquasir-quasih}, \cite{VasPseudoH}, \cite{VasTm2018} and \cite{Rodriguez2020} where increasingly larger families of symbols were considered. We also refer to \cite{BVQuasiEllipticI}, \cite{CuckovicLouhichi2008} and \cite{QSpseudoconvex}, where similar notions have been used. All these works study the so-called quasi-homogeneous and pseudo-homogeneous symbols of different sorts. See the observations before Remark~\ref{rmk:(k,j)-quasi} for detailed definitions. It is easy to show that, in most cases, such symbols from previous works are invariant under the group $\T^m$, the center of the group $\U(\kb)$ described above. Such symbols were used before to build commutative Banach algebras generated by Toeplitz operators, which are not $C^*$-algebras in the generic case.

Nevertheless, in these previous works some questions were not answered and some problems remained without being considered.

On one hand, even though symbols invariant under the group $\T^m$ have been used before, only a small family of them has actually been studied. The Toeplitz operators for the whole collection of $\T^m$-invariant symbols, given by $L^\infty(\B^n)^{\T^m}$, is far from  understood. We observe that, as proved in Corollary~\ref{cor:pikapb_irreducible}, the $C^*$-algebras generated by Toeplitz operators with symbols in $L^\infty(\B^n)^{\T^m}$ are not commutative when $m < n$, on every weighted Bergman space. So one might expect to be a very difficult problem to describe the properties of general Toeplitz operators with symbols in $L^\infty(\B^n)^{\T^m}$. For example, as noted in Subsection~\ref{subsec:commutative}, Schur's Lemma is not enough to study the algebra or its individual operators.

On the other hand, the previous works have exhibited, as explained above, some interesting commutative Banach non-$C^*$ algebras generated by Toeplitz operators. Note that, for every subgroup $H$ of $\U(n)$, the set of symbols $L^\infty(\B^n)^H$ yields Toeplitz operators that generate Banach algebras which are necessarily $C^*$. Hence, it might appear at first that for the known examples of commutative Banach non-$C^*$ algebras there is no group $H$ that could characterize them through invariance. This argument could be further supported by the fact that the $\T^m$-invariant symbols considered so far (quasi-homogeneous and pseudo-homogeneous symbols from the previous references) cannot be defined in terms of an $H$-invariance property for any group $H$.

\subsection{Quasi-homogeneous symbols, groups and representation theory: main results}
\label{subsec:quasi_and_groups}
The goal of this work is to provide some solutions to the unresolved matters described above. To achieve this, we study Toeplitz operators with arbitrary $\T^m$-invariant symbols. We introduce, as well, a quite general notion of quasi-homogeneous symbols defined using groups, that basically includes those considered before, and study the corresponding Toeplitz operators.

Our main tool is representation theory of compact groups. We use the well known fact that the unitary group $\U(n)$ has a natural representation $\pi_\lambda$ on the weighted Bergman space $\cA_\lambda^2(\B^n)$ (see Section~\ref{sec:U(n)rep}). In particular, we use throughout this work that, for every closed subgroup $H$ of $\U(n)$, the Toeplitz operators with symbols in $L^\infty(\B^n)^H$ are intertwining for the representation $\pi_\lambda|_H$ obtained by restricting $\pi_\lambda$ from $\U(n)$ to $H$ (see Corollary~\ref{cor:HinvariantToeplitz}). Intertwining operators preserve the terms of the corresponding isotypic decompositions. This is an important representation theoretic tool we will use as well throughout our work.

We consider the $C^*$-algebras generated by Toeplitz operators with symbols in $L^\infty(\B^n)^{\T^m}$. The non-commutativity of these $C^*$-algebras is obtained by exhibiting some finite dimensional irreducible representations in Corollary~\ref{cor:pikapb_irreducible}. We compute the isotypic decomposition of the Bergman spaces $\cA^2_\lambda(\B^n)$ for the representation $\pi_\lambda|_{\T^m}$ (see Proposition~\ref{prop:isotypicTm}). Then, we use the intertwining property of Toeplitz operators with $\T^m$-invariant symbols to prove they are block diagonal for such isotypic decomposition (see Proposition~\ref{prop:intertwiningTm}). This block diagonal structure is explained in \eqref{eq:block_lambdaT} and \eqref{eq:block_lambdaa}.

By an averaging process that has been used before in similar setups (see for example \cite{DOQJFA} and \cite{ZorbRadial2002}), we assign to every symbol $a \in L^\infty(\B^n)^{\T^m}$ a corresponding quasi-radial symbol $\widehat{a}$. It turns out that the Toeplitz operators $T^{(\lambda)}_{\widehat{a}}$ and $T^{(\lambda)}_a$ both preserve the isotypic decomposition of $\pi_\lambda|_{\T^m}$. As explained through equations \eqref{eq:gamma_lambdaa} and \eqref{eq:block_lambdaa}, $T^{(\lambda)}_{\widehat{a}}$ acts diagonally and $T^{(\lambda)}_a$ acts block diagonally on such decomposition. This leads us to Theorem~\ref{thm:traceToeplitzTm}, where we prove that the traces of the operators $T^{(\lambda)}_{\widehat{a}}$ and $T^{(\lambda)}_a$ are the same on the blocks for the isotypic decomposition of $\pi_\lambda|_{\T^m}$. This allows us to obtain in Theorem~\ref{thm:traceToeplitzTm_integral} a formula for the traces of the blocks of a Toeplitz operator $T^{(\lambda)}_a$ with an arbitrary symbol $a \in L^\infty(\B^n)^{\T^m}$. Such formula is presented as an integral involving the symbol $a$. Furthermore, we describe in Theorem~\ref{thm:traceToeplitzT_SON} the asymptotic behavior and the closure of the space of sequences of traces of blocks for the Toeplitz operators $T^{(\lambda)}_a$ for $a \in L^\infty(\B^n)^{\T}$ (case $m = 1$). This last result applies a corresponding result from \cite{BHVRadial} for radial symbols. This shows the usefulness of averaging constructions to obtain information for $\T^m$-invariant symbols from results for quasi-radial symbols. This sort of argument is of a Lie theoretic nature through the use of the Haar measure of the group $\U(\kb)$.

Next, we consider the subgroup $\U(\kb,j,\T)$ of $\U(\kb)$ obtained by replacing the $j$-th factor $\U(k_j)$ with its center $\T$. See Section~\ref{sec:quasi} for further details. We introduce in this section the symbols that are $\U(\kb,j,\T)$-invariant, which we call $(\kb,j)$-quasi-radial quasi-homogeneous (see Definition~\ref{def:kbj-quasi}). They are easy to characterize (see Lemma~\ref{lem:kbj-quasi}) and turn out to generalize most of the previous notions of quasi-homogeneous and pseudo-homogeneous symbols from the references mentioned above (see Remark~\ref{rmk:(k,j)-quasi} and the preceding observations). Since $\U(\kb,j,\T) \supset \T^m$, the Toeplitz operators with $(\kb,j)$-quasi-radial quasi-homogeneous symbols intertwine the isotypic decomposition of $\pi_\lambda|_{\T^m}$. Furthermore, we prove that, up to isomorphism, such Toeplitz operators admit a tensor product representation where all the factors are the identity map except for a single factor (see Proposition~\ref{prop:EndU(kb,j,T)-tensorproduct}). The proof uses Schur's Lemma, but goes beyond it by making use of Proposition~\ref{prop:Pkapb-tensorproduct} which is based on the characterization of the tensor product of irreducible representations. As a consequence we obtain in Theorem~\ref{thm:quasisymbols_commToeplitz} the mutual commutativity of Toeplitz operators with $(\kb,j)$-quasi-radial quasi-homogeneous symbols when the value of $j$ varies.

Theorem~\ref{thm:quasisymbols_commToeplitz} turns out to be a quite strong generalization of similar results obtained in other works mentioned above. In these previous works, the results are obtained through particular computations for each specific type of symbols considered. In our case, the commutativity is a consequence of the invariance under suitable groups $\U(\kb,j,\T)$, and holds for any symbol once such invariance is assumed.

As an application, Corollary~\ref{cor:quasi_commBanach} proves the existence of commutative Banach non-$C^*$ algebras generated by Toeplitz operators using our $(\kb,j)$-quasi-radial quasi-homogeneous symbols. These algebras include all those obtained in the previous literature using quasi-radial quasi-homogeneous symbols, but ours are quite more general because we consider more general symbols.

Besides the fact that we obtain a broader family of commutative Banach algebras, there are a couple of other advantages in our approach. In the first place, our proof is uniform and does not depend on a particular type of symbol as it does in the previous literature. Second, our proofs and constructions show that the commutative Banach non-$C^*$ algebras considered in Section~\ref{sec:quasi}, as well as those from the references mentioned above, can indeed be characterized using groups through an invariance condition for suitable~symbols.

Another advantage of our representation theoretic approach is how relatively easy it is to obtain the desired results through the use of isotypic decompositions. However, in the past this has also been tied to the difficulty to obtain explicit formulas that describe the Toeplitz operators with special symbols. Nevertheless, we are able to obtain in Section~\ref{sec:acting_on_monomials} a very detailed description of the block structure of the Toeplitz operators with our $(\kb,j)$-quasi-radial quasi-homogeneous symbols. More precisely, Proposition~\ref{prop:EndU(kb,j,T)_matrixcoef} proves that the blocks of such Toeplitz operators have a finer block structure given by a matrix that repeats along the diagonal (see Remark~\ref{rmk:Tblockdiagonalconst}). Furthermore, we obtain in Theorems~\ref{thm:Tkappaj_withf_forToeplitz} and \ref{thm:Tkappaj_withg_forToeplitz}, and Corollaries~\ref{cor:Tkappaj_withf_forToeplitz} and \ref{cor:Tkappaj_withg_forToeplitz}, explicit formulas for the Toeplitz operators through their action on the monomial basis. These formulas are given as integrals involving the corresponding symbols. This allows us to obtain in Remark~\ref{rmk:Tblockdiagonaldescription} an explicit description, up to unitary equivalence, of all Toeplitz operators with $(\kb,j)$-quasi-radial quasi-homogeneous symbols in terms of integral formulas involving such symbols.

From the representation theoretic viewpoint, the proofs of the main results described above ultimately depend on Haar measure on compact groups, Schur's Lemma, the characterization of tensor products of irreducible representations, the computation of isotypic decompositions and the general properties of intertwining operators. To the best of our knowledge, this is the first time this variety of techniques is used to obtain such kind of results dealing with Toeplitz operators.

As for the distribution of this work, Sections~\ref{sec:Bergman} and \ref{sec:U(n)rep} introduce the preliminary material from analysis and representation theory, respectively. We compute in Section~\ref{sec:isotypic} the isotypic decompositions used in the rest of the work. Sections~\ref{sec:Tm_and_quasi-radial}, \ref{sec:quasi} and \ref{sec:acting_on_monomials} contain the main results and their proofs as described above.

Finally, the author wishes to thank the anonymous reviewer whose comments helped him to greatly improve the presentation of this work.

\section{Bergman spaces and Toeplitz operators}\label{sec:Bergman}
On the unit ball $\B^n$ of $\C^n$ consider the Lebesgue measure $\dif v(z)$. For every $\lambda > -1$ we denote the weighted measure
\[
    \dif v_\lambda(z)
        = c_\lambda (1-|z|^2)^\lambda \dif v(z)
\]
where $c_\lambda$ is a normalizing constant. More precisely, we have
\[
    c_\lambda = \frac{\Gamma(n+\lambda+1)}{\pi^n\Gamma(\lambda+1)}.
\]
These allow us to define for every $\lambda > -1$ the weighted Bergman space $\cA^2_\lambda(\B^n)$ which consist of the holomorphic functions on $\B^n$ that belong to $L^2(\B^n,v_\lambda)$. This is a well known closed subspace whose orthogonal projection, called the Bergman projection, is given by
\begin{align*}
    B_\lambda : L^2(\B^n,v_\lambda) &\rightarrow \cA^2_\lambda(\B^n) \\
    B_\lambda(f)(z) &= \int_{\B^n}
        \frac{f(w)\dif v_\lambda(w)}{(1 - z\cdot\overline{w})^{n+\lambda+1}}.
\end{align*}
In particular, the space $\cA^2_\lambda(\B^n)$ is a reproducing kernel Hilbert space whose kernel is $K_\lambda(z,w) = (1-z\cdot\overline{w})^{-(n+\lambda+1)}$.

For every $a \in L^\infty(\B^n)$ the Toeplitz operator $T^{(\lambda)}_a$ with symbol $a$ acting on the weighted Bergman space $\cA^2_\lambda(\B^n)$ is given by the following expression
\[
    T^{(\lambda)}_a(f)(z) = B_\lambda(af)(z)
    = \int_{\B^n}
        \frac{a(w)f(w)\dif v_\lambda(w)}{(1 - z\cdot\overline{w})^{n+\lambda+1}}.
\]

\section{The representations of $\U(n)$ on Bergman spaces over $\B^n$}\label{sec:U(n)rep}
As usual, we will denote by $\U(n)$ the group of unitary matrices acting on $\C^n$. The canonical maximal torus of $\U(n)$ is the subgroup $\T^n$ of diagonal matrices, whose elements will be simply written as $t = (t_1, \dots, t_n)$ where $t_j \in \T$ for every $j = 1, \dots, n$. In particular, for every $t \in \T^n$ and $z \in \B^n$, the $\T^n$-action is given by
\[
    t\cdot z = (t_1 z_1, \dots, t_n z_n).
\]

There is a natural and very general representation of $\U(n)$ on functions $f$ defined on $\B^n$ given by the expression
\[
    A\cdot f = f\circ A^{-1},
\]
for every $A \in \U(n)$.

In the first place, this yields a unitary representation $\pi_\lambda : \U(n) \rightarrow U(\cA^2_\lambda(\B^n))$ on the weighted Bergman space $\cA^2_\lambda(\B^n)$ given by
\[
    \pi_\lambda(A)(f) = f\circ A^{-1}.
\]
That the map $\pi_\lambda(A)$ is unitary, for every $A \in \U(n)$, follows from the fact that the $\U(n)$-action on $\B^n$ preserves the measure $v_\lambda$. Furthermore, this representation is continuous in the strong operator topology. Equivalently, the corresponding action map $\U(n) \times \cA^2_\lambda(\B^n) \rightarrow \cA^2_\lambda(\B^n)$ is continuous.

Similarly, we have an action of the group $\U(n)$ on the space of symbols $L^\infty(\B^n)$ by using the same expression
\[
    A\cdot a = a\circ A^{-1}.
\]
This action is isometric since $\|A\cdot a\|_\infty = \|a\|_\infty$ for every $A \in \U(n)$ and $a \in L^\infty(\B^n)$.

On the other hand, the unitary representation $\pi_\lambda$ on $\cA^2_\lambda(\B^n)$ induces an action of $\U(n)$ on the algebra $B(\cA^2_\lambda(\B^n))$ of bounded operators on $\cA^2_\lambda(\B^n)$ given by
\[
    T \mapsto \pi_\lambda(A)T\pi_\lambda(A)^{-1},
\]
for every $T \in B(\cA^2_\lambda(\B^n))$ and $A \in \U(n)$.

In the case of Toeplitz operators the following easy to prove result yields a relationship between two of these actions.

\begin{lemma}\label{lem:U(n)equiv-Toeplitz}
    For every symbol $a \in L^\infty(\B^n)$ and on every weighted Bergman space $\cA^2_\lambda(\B^n)$ we have
    \[
        \pi_\lambda(A) T^{(\lambda)}_a
        \pi_\lambda(A)^{-1}
        = T^{(\lambda)}_{A\cdot a},
    \]
    for every $A \in \U(n)$. In other words the map $L^\infty(\B^n) \rightarrow B(\cA^2_\lambda(\B^n))$ given by $a \mapsto T^{(\lambda)}_a$ is $\U(n)$-equivariant.
\end{lemma}
\begin{proof}
    For every $f,g \in \cA^2_\lambda(\B^n)$ we have
    \begin{align*}
        \langle \pi_\lambda(A) T^{(\lambda)}_a
        \pi_\lambda(A)^{-1} f,g\rangle_\lambda
            &= \langle T^{(\lambda)}_a
                \pi_\lambda(A)^{-1}f, \pi_\lambda(A)^{-1}g\rangle_\lambda
                \\
            &= \langle a
                \pi_\lambda(A)^{-1}f, \pi_\lambda(A)^{-1}g\rangle_\lambda
                \\
            &= \langle
            \pi_\lambda(A)^{-1}((A\cdot a) f), \pi_\lambda(A)^{-1}g\rangle_\lambda
            \\
            &= \langle (A\cdot a) f,
            g\rangle_\lambda
            = \langle T^{(\lambda)}_{A\cdot a} f,
            g\rangle_\lambda
    \end{align*}
\end{proof}

We recall that for a subgroup $H$ of $\U(n)$ a symbol $a \in L^\infty(\B^n)$ is called $H$-invariant if we have $a = a\circ A$ for every $A \in H$. We will denote by $L^\infty(\B^n)^H$ the space of all essentially bounded $H$-invariant symbols on $\B^n$.

The corresponding notion for operators is given by the so-called intertwining operators. More precisely, for a subgroup $H$ of $\U(n)$ we say that an operator $T \in B(\cA^2_\lambda(\B^n))$ intertwines the restriction $\pi_\lambda|_H$ if the condition
\[
    T\pi_\lambda(A) = \pi_\lambda(A)T
\]
holds for every $A \in H$. We also say that $T$ is $H$-equivariant. We will denote by $\End_H(\cA^2_\lambda(\B^n))$ the algebra of all such intertwining, or $H$-equivariant, operators. Since the representation $\pi_\lambda$ is unitary it follows easily that $\End_H(\cA^2_\lambda(\B^n))$ is a $C^*$-algebra. Furthermore, it is well known that $\End_H(\cA^2_\lambda(\B^n))$ is a von Neumann algebra.

As an immediate consequence of Lemma~\ref{lem:U(n)equiv-Toeplitz} we obtain the next result. It also uses the well known fact that the assignment $a \mapsto T^{(\lambda)}_a$ is injective. From now on, for every subgroup $H$ of $\U(n)$, we will denote by $\cT^{(\lambda)}(L^\infty(\B^n)^H)$ the $C^*$-algebra generated by the Toeplitz operators with symbols in $L^\infty(\B^n)^H$.

\begin{corollary}\label{cor:HinvariantToeplitz}
    For a closed subgroup $H$ of $\U(n)$ and for every symbol $a \in L^\infty(\B^n)$ the following conditions are equivalent for every $\lambda > -1$.
    \begin{enumerate}
      \item The symbol $a$ is $H$-invariant, i.e.~$a$ belongs to $L^\infty(\B^n)^H$.
      \item The Toeplitz operator $T^{(\lambda)}_a$ intertwines the restriction $\pi_\lambda|_H$, i.e.~$T^{(\lambda)}_a$ belongs to $\End_H(\cA^2_\lambda(\B^n))$.
    \end{enumerate}
    In particular, we have $\cT^{(\lambda)}(L^\infty(\B^n)^H) \subset \End_H(\cA^2_\lambda(\B^n))$.
\end{corollary}

A remarkable fact is that Toeplitz operators constitute a large subset in the set of all operators. This holds even when $H$-equivariance is being considered. That is the content of the next result which is a consequence of Proposition~6.2 from \cite{DOQJFA}.

\begin{proposition}\label{prop:HToeplitz_density}
    For every closed subgroup $H$ of $\U(n)$ and for every $\lambda > -1$, the $C^*$-algebra $\cT^{(\lambda)}(L^\infty(\B^n)^H)$ is dense in $\End_H(\cA^2_\lambda(\B^n))$ with respect to the strong operator topology.
\end{proposition}

As in the previous results, consider a closed subgroup $H$ of $\U(n)$. Let us denote by $\mu$ a Haar measure for the group $H$. Since $H$ is compact, $\mu$ is bi-invariant and finite. Hence, we will assume for simplicity that $\mu$ is a probability measure. For every $\lambda > -1$ and for every $T \in B(\cA^2_\lambda(\B^n))$ we consider the operator $\widehat{T}$ defined by
\begin{equation}\label{eq:widehatT}
    \langle \widehat{T}f, g\rangle_\lambda =
    \int_H \langle \pi_\lambda(A) T \pi_\lambda(A)^{-1} f, g\rangle_\lambda \dif \mu(A),
\end{equation}
for every $f, g \in \cA^2_\lambda(\B^n)$. The next result is an easy consequence of the definition provided by \eqref{eq:widehatT}.

\begin{lemma}\label{lem:widehatT}
    Let $H$ be a closed subgroup of $\U(n)$. Then, for every $\lambda > -1$, the assignment $T \mapsto \widehat{T}$ given by \eqref{eq:widehatT} is a linear map that satisfies the following properties.
    \begin{enumerate}
      \item For every $T \in B(\cA^2_\lambda(\B^n))$, the operator $\widehat{T}$ belongs to $\End_H(\cA^2_\lambda(\B^n))$. In other words, $\widehat{T}$ is a bounded intertwining operator for the restriction $\pi_\lambda|_H$. Furthermore, we have $\|\widehat{T}\| \leq \|T\|$.
      \item If $T \in \End_H(\cA^2_\lambda(\B^n))$, then $\widehat{T} = T$.
    \end{enumerate}
\end{lemma}

Following a similar construction, for every $a \in L^\infty(\B^n)$ we denote
\begin{equation}\label{eq:widehata}
    \widehat{a}(z) = \int_H a \circ A^{-1}(z) \dif \mu(A),
\end{equation}
for every $z \in \B^n$. Then, it is immediately seen that $\widehat{a} \in L^\infty(\B^n)^H$. Finally, we have the next easy to prove fact (see for example \cite{DOQJFA}).

\begin{lemma}\label{lem:widehatTa}
    With the previous notation and for a closed subgroup $H$ of $\U(n)$ we have
    \[
        \widehat{T_a^{(\lambda)}} = T_{\widehat{a}}^{(\lambda)},
    \]
    for every $\lambda > -1$ and for every $a \in L^\infty(\B^n)$.
\end{lemma}

\begin{remark}
    Note that the definitions provided by \eqref{eq:widehatT} and \eqref{eq:widehata} depend on the subgroup $H$ under consideration. Even though such subgroup is not indicated in the expressions $\widehat{T}$ and $\widehat{a}$, it will be either specified in advance or easy to determine from the context in the rest of this work.
\end{remark}

\section{Isotypic decompositions}\label{sec:isotypic}
Let $H$ be a closed subgroup of $\U(n)$ and $V$ an irreducible $H$-module. The isotypic component associated to $V$ for the restriction $\pi_\lambda|_H$ is the sum of all the $H$-submodules of $\cA^2_\lambda(\B^n)$ that are isomorphic to the $H$-module $V$. It is well known (see \cite{BtD}) that the compactness of $H$ implies that $\cA^2_\lambda(\B^n)$ is the Hilbert direct sum of the isotypic components of $\pi_\lambda|_H$. Such direct sum is called the isotypic decomposition of the restriction $\pi_\lambda|_H$. We say that the isotypic decomposition, as well as the representation $\pi_\lambda|_H$ or the $H$-action, is multiplicity-free if each isotypic component is in fact irreducible.

We now describe the isotypic decompositions for the representation $\pi_\lambda$ of $\U(n)$ and its restriction to some of its subgroups. These decompositions will be given in terms of the space of complex holomorphic polynomials on $\C^n$, which we will denote by $\cP(\C^n)$, and the subspaces of homogeneous polynomials of degree $\kappa$, that we will be denote by $\cP_\kappa(\C^n)$ for every $\kappa \in \N$. The elements of $\cP(\C^n)$ can be restricted to $\B^n$ and thus can be considered to be elements of the Bergman spaces. In other words, we have $\cP(\C^n) \subset \cA^2_\lambda(\B^n)$, for every $\lambda > -1$. We will use this fact without further mention in the rest of this work.

We first consider the subgroup $\T^n$.

\begin{proposition}\label{prop:isotypicTn}
    For every $\lambda > -1$, the isotypic decomposition of the restriction $\pi_\lambda|_{\T^n}$ is given by
    \[
        \cA^2_\lambda(\B^n)
            = \bigoplus_{\alpha \in \N^n} \C z^\alpha.
    \]
    In particular, the isotypic decomposition is multiplicity-free.
\end{proposition}
\begin{proof}
    The Hilbert direct sum follows from the well known fact that the monomials $z^\alpha$, for $\alpha \in \N^n$, are an orthogonal basis of the weighted Bergman space $\cA^2_\lambda(\B^n)$. On the other hand, for any given $\alpha \in \N^n$ we have
    \[
        t\cdot z^\alpha = t^{-\alpha} z^\alpha
    \]
    for every $t \in \T^n$. Hence, the $1$-dimensional space $\C z^\alpha$ is an irreducible $\T^n$-module with character $t \mapsto t^{-\alpha}$. In particular, these spaces are inequivalent to each other as $\T^n$-modules.
\end{proof}

We now state the corresponding result for the group $\U(n)$.

\begin{proposition}\label{prop:isotypicU(n)}
    For every $\lambda > -1$, the isotypic decomposition of the unitary representation $\pi_\lambda$ of $\U(n)$ is given by
    \[
        \cA^2_\lambda(\B^n)
            = \bigoplus_{\kappa \in \N}
                \cP_\kappa(\C^n).
    \]
    Furthermore, the isotypic decomposition of the $\U(n)$-action is multiplicity-free.
\end{proposition}
\begin{proof}
    The Hilbert direct sum follows from Proposition~\ref{prop:isotypicTn}. On the other hand, it is well known that $\cP_\kappa(\C^n)$ is an irreducible $\U(n)$-module (see \cite{GW} and \cite{QRadial}). For $n \geq 2$, the spaces $\cP_\kappa(\C^n)$ all have different dimensions and so are inequivalent as $\U(n)$-modules. The case $n = 1$ is already covered by Proposition~\ref{prop:isotypicTn} since $\U(1) = \T$ acting on $\B^1$.
\end{proof}

Some other subgroups of $\U(n)$ will be of interest to us. Particularly, those associated with partitions of $n$.

Let $\kb = (k_1, \dots, k_m) \in \Z_+^m$ be a partition of $n$. In other words, $k_1 + \dots + k_m = n$, and each $k_j$ is a positive integer. Such a partition induces a decomposition
\[
    \C^n = \C^{k_1}\times\dots\times\C^{k_m}.
\]
This yields a corresponding coordinate decomposition such that for $z \in \C^n$ we have
\[
    z = (z_{(1)}, \dots, z_{(m)})
\]
where $z_{(j)} \in \C^{k_j}$, for every $j=1,\dots,m$. For simplicity, we can always assume that $k_1 \leq k_2 \leq \dots \leq k_m$. In particular, we can associate to the partition $\kb$ the smallest non-negative integer $h$ such that $k_j = 1$ if and only if $j \leq h$.

In the rest of this work we consider a fixed partition $\kb \in \Z_+^m$ of $n$ together with its associated integer $h$ defined above. Also, for every $\beta \in \Z^\ell_+$ we use the standard notation $|\beta| = \beta_1 + \dots + \beta_\ell$.

For our fixed partition $\kb$ we will consider the following product of unitary groups
\begin{align*}
    \U(\kb) &= \U(k_1)\times\dots\times\U(k_m) \\
        &= \T^h \times \U(k_{h+1})\times\dots\times\U(k_m).
\end{align*}
Observe that the group $\U(\kb)$ can be canonically realized as a block diagonal subgroup of $\U(n)$. We will use this realization in the rest of this work.

On the other hand, for every $\kapb \in \N^m$ we will denote by $\cP_\kapb(\C^n)$ the linear span of the monomials $z^\alpha$ such that $|\alpha_{(j)}| = \kappa_j$ for all $j=1, \dots, m$. We have the following easy to prove alternative description of $\cP_\kapb(\C^n)$.

\begin{lemma}\label{lem:cPkapb}
    Let $\kb \in \Z_+^m$ be a partition of $n$ and let $\kapb \in \N^m$ be given. Then, a polynomial $p(z)$ belongs to $\cP_\kapb(\C^n)$ if and only if for every $j = 1, \dots, m$, and every fixed $(z_{(1)}, \dots, z_{(j-1)}, z_{(j+1)}, \dots, z_{(m)})$ the polynomial in $w \in \C^{k_j}$ given by
    \[
         w \mapsto p(z_{(1)}, \dots, z_{(j-1)}, w, z_{(j+1)}, \dots, z_{(m)})
    \]
    is homogeneous of degree $\kappa_j$. In particular, $\cP_\kapb(\C^n)$ is an invariant subspace for the restriction $\pi_\lambda|_{\U(\kb)}$.
\end{lemma}

The next realization of the space $\cP_\kapb(\C^n)$ as an irreducible $\U(\kb)$-module will be very useful. This result depends on the characterization of irreducible representations of a product of groups as an outer tensor product of irreducible representations. We refer to \cite{BtD} for further details and definitions.

\begin{proposition}\label{prop:Pkapb-tensorproduct}
    Let $\kb \in \Z_+^m$ be a partition of $n$. Then, for every $\kapb \in \N^m$ the linear transformation
    \begin{align*}
        U_\kapb : \bigotimes_{j=1}^m \cP_{\kappa_j}(\C^{k_j})
            &\rightarrow \cP_\kapb(\C^n) \\
        p_1\otimes \dots \otimes p_m &\mapsto
            q(z) = p_1(z_{(1)})\dots p_m(z_{(m)})
    \end{align*}
    is an isomorphism of $\U(\kb)$-modules, where the domain carries the outer tensor product structure over $\U(\kb) = \U(k_1) \times \dots \times \U(k_m)$. In particular, $\cP_\kapb(\C^n)$ is an irreducible $\U(\kb)$-module.
\end{proposition}
\begin{proof}
    That $U_\kapb$ is well defined is a consequence of the usual properties of the tensor product. On the other hand, $U_\kapb$ maps
    \[
        z_{(1)}^{\alpha_{(1)}} \otimes \dots
        \otimes z_{(m)}^{\alpha_{(m)}} \mapsto
        z^\alpha,
    \]
    where $\alpha \in \N^n$ satisfies $|\alpha_{(j)}| = \kappa_j$ for all $j=1, \dots, m$. In other words, it maps a basis onto a basis, and so $U_\kapb$ is an isomorphism. Furthermore, it is clear that the map $U_\kapb$ is $\U(\kb)$-equivariant.

    Since the $\U(\kb)$-module $\bigotimes_{j=1}^m \cP_{\kappa_j}(\C^{k_j})$ is an outer tensor product of irreducible modules it follows that it is itself irreducible (see \cite{BtD}). Hence, $\cP_\kapb(\C^n)$ is an irreducible $\U(\kb)$-module as well.
\end{proof}

In the rest of this work, and for every $\kapb \in \N^m$ as above, we will freely use the identification of $\cP_\kapb(\C^n)$ with the tensor product given by Proposition~\ref{prop:Pkapb-tensorproduct}.

We now describe the isotypic decomposition for the group $\U(\kb)$.

\begin{proposition}\label{prop:isotypicU(kb)}
    For every $\lambda > -1$, the isotypic decomposition of the restriction $\pi_\lambda|_{\U(\kb)}$ is given by
    \[
        \cA^2_\lambda(\B^n)
            = \bigoplus_{\kapb \in \N^m} \cP_\kapb(\C^n).
    \]
    Furthermore, this isotypic decomposition is multiplicity-free.
\end{proposition}
\begin{proof}
    From the definitions we have for every $\ell \in \N$
    \[
        \cP_\ell(\C^n)
            = \bigoplus_{\kapb \in \N^m, |\kapb| = \ell} \cP_\kapb(\C^n)
    \]
    and so the Hilbert direct sum from the statement follows from Proposition~\ref{prop:isotypicU(n)} and the orthogonality of the monomial basis. Also, Proposition~\ref{prop:Pkapb-tensorproduct} shows that such Hilbert direct sum is a sum of irreducible $\U(\kb)$-modules.

    Furthermore, by the properties of the outer tensor product of actions (see \cite{BtD}) and Proposition~\ref{prop:Pkapb-tensorproduct}, we have $\cP_\kapb(\C^n) \simeq \cP_{\kapb'}(\C^n)$ as modules over $\U(\kb)$ if and only if $\kapb = \kapb'$, and so the result follows.
\end{proof}

\begin{remark}\label{rmk:decomp_into_cPkapb}
    The Hilbert direct sum decomposition from Proposition~\ref{prop:isotypicU(kb)} was used in \cite{VasTm2018} (see equation (3.1) therein). However, in \cite{VasTm2018} it is not considered the fact that this is a multiplicity-free isotypic decomposition for a suitable subgroup in $\U(n)$, as we have proved in Proposition~\ref{prop:isotypicU(kb)} for the subgroup~$\U(\kb)$.
\end{remark}

For our given partition $\kb \in \Z_+^m$ and for every $j=1,\dots,m$, we will denote by $I_{(j)}$ the identity operator in $\C^{k_j}$, and we will use this notation in the rest of this work. We consider the subgroup of $\U(\kb)$ defined by
\[
    \T^m = \{ (t_1 I_{(1)}, \dots, t_m I_{(m)})
            \mid t_1, \dots, t_m \in \T\}.
\]
In other words, $\T^m$ is precisely the center of $\U(\kb)$. Furthermore, with respect to the natural inclusion $\U(\kb) \subset \U(n)$ we observe that $\T^m \subset \U(\kb) \cap \T^n$, where $\T^n$ is considered as the subgroup of diagonal matrices of $\U(n)$. In particular, $\T^m \subset \T^n$ by these identifications.

From now on, together with our fixed choice of the partition $\kb$, we also consider fixed the subgroups $\U(\kb)$ and $\T^m$ obtained through the use of $\kb$ with the previous constructions.

We now obtain the isotypic decomposition corresponding to $\T^m$.

\begin{proposition}\label{prop:isotypicTm}
    For every $\lambda > -1$, the isotypic decomposition of the restriction $\pi_\lambda|_{\T^m}$ is given by
    \[
        \cA^2_\lambda(\B^n)
            = \bigoplus_{\kapb \in \N^m} \cP_\kapb(\C^n).
    \]
\end{proposition}
\begin{proof}
    The Hilbert direct sum was already obtained in Proposition~\ref{prop:isotypicU(kb)}.

    On the other hand, for a given $\kapb \in \N^m$, if $p \in \cP_\kapb(\C^n)$, then we have
    \[
        t\cdot p(z)
        = t_1^{-\kappa_1} \dots t_m^{-\kappa_m} p(z)
    \]
    for every $t \in \T^m$. This follows easily from Lemma~\ref{lem:cPkapb}. Hence, $\cP_\kapb(\C^n)$ is the space of polynomials that transform under the $\T^m$-action according to the character defined by $-\kapb$. We conclude that $\cP_\kapb(\C^n)$ is the isotypic component associated to the character defined by $-\kapb$ and so the result follows.
\end{proof}

\begin{remark}\label{rmk:isotypicTm}
    We observe that the isotypic decomposition from Proposition~\ref{prop:isotypicTm} is multiplicity-free if and only if $\dim \cP_\kapb(\C^n) = 1$ for every $\kapb \in \N^m$, because $\T^m$ is an Abelian group. And this holds if and only if $m = n$, which occurs when the partition $\kb$ of $n$ is precisely $(1, \dots, 1) \in \Z_+^n$. It is only in this case that we obtain a multiplicity-free isotypic decomposition for the restriction~$\pi_\lambda|_{\T^m}$.
\end{remark}

We recall that Schur's Lemma implies that any intertwining operator preserves the isotypic components and that, in the multiplicity-free case, it acts on them by a constant multiple of the identity. Hence, as a consequence of Proposition~\ref{prop:isotypicU(kb)} we obtain the next result.

\begin{proposition}\label{prop:intertwiningU(kb)}
    Let $\lambda > -1$ and $\kb \in \Z_+^m$ be given. Then, the map defined by
    \begin{align*}
        \ell^\infty(\N^m) &\rightarrow \End_{\U(\kb)}
        (\cA^2_\lambda(\B^n)) \\
        (c_\kapb)_{\kapb \in \N^m} &\mapsto
            \bigoplus_{\kapb \in \N^m}
                    c_\kapb I_\kapb
    \end{align*}
    is an isomorphism of $C^*$-algebras, where $I_\kapb$ is the identity map over~$\cP_\kapb(\C^n)$.
\end{proposition}

As noted in Proposition~\ref{prop:isotypicTm}, the isotypic decomposition for the restriction $\pi_\lambda|_{\T^m}$ is, in general, not multiplicity-free. In this case we obtain the following result.

\begin{proposition}\label{prop:intertwiningTm}
    Let $\lambda > -1$ and $\kb \in \Z_+^m$ be given. Let us consider the corresponding subgroup $\T^m$ of $\U(n)$. Then, every $T \in \End_{\T^m}(\cA^2_\lambda(\B^n))$ satisfies
    \[
        T(\cP_\kapb(\C^n)) \subset \cP_\kapb(\C^n)
    \]
    for every $\kapb \in \N^m$. Furthermore, the map defined by
    \begin{align*}
        \End_{\T^m}(\cA^2_\lambda(\B^n))
            &\rightarrow
            \bigoplus_{\kapb \in \N^m}
            \End(\cP_\kapb(\C^n)) \\
        T &\mapsto \bigoplus_{\kapb \in \N^m}
                    T|_{\cP_\kapb(\C^n)}
    \end{align*}
    is an isomorphism of $C^*$-algebras.
\end{proposition}
\begin{proof}
    If $T \in \End_{\T^m}(\cA^2_\lambda(\B^n))$, then Schur's Lemma and Proposition~\ref{prop:isotypicTm} imply that
    \[
        T(\cP_\kapb(\C^n)) \subset \cP_\kapb(\C^n)
    \]
    for every $\kapb \in \N^m$. In particular, we have
    \[
        T = \bigoplus_{\kapb \in \N^m}
            T|_{\cP_\kapb(\C^n)}.
    \]
    Hence, the map in the statement is an injective and well-defined homomorphism of $C^*$-algebras.

    On the other hand, by definition of the direct sum of $C^*$-algebras, if we choose
    \[
        (T_\kapb)_{\kapb \in \N^m} \in
        \bigoplus_{\kapb \in \N^m}
            \End(\cP_\kapb(\C^n)),
    \]
    then the operator defined by
    \[
        T = \bigoplus_{\kapb \in \N^m} T_\kapb
    \]
    is a bounded operator of $\cA^2_\lambda(\B^n)$. Furthermore, as noted in the proof of Proposition~\ref{prop:isotypicTm}, the $\T^m$-action on $\cP_\kapb(\C^n)$ is given by multiplication by a character. In particular, the operator $T$ intertwines the $\T^m$-action and so it belongs to the algebra $\End_{\T^m}(\cA^2_\lambda(\B^n))$. Since we have $T|_{\cP_\kapb(\C^n)} = T_\kapb$, for every $\kapb \in \N^m$, this proves the surjectivity of the homomorphism of $C^*$-algebras.
\end{proof}

The following is an immediate consequence of Propositions~\ref{prop:intertwiningU(kb)} and \ref{prop:intertwiningTm}. Note that the conclusion occurs in a situation where the group $\T^m$ is precisely the center of $\U(\kb)$. In symbols, we have $\T^m = Z(\U(\kb))$.
\begin{corollary}\label{cor:centerEndTm}
    The center of the $C^*$-algebra $\End_{\T^m}(\cA^2_\lambda(\B^n))$ is the $C^*$-algebra $\End_{\U(\kb)} (\cA^2_\lambda(\B^n))$. In other words, we have
    \[
        \End_{\U(\kb)} (\cA^2_\lambda(\B^n)) =
        Z(\End_{\T^m}(\cA^2_\lambda(\B^n))).
    \]
\end{corollary}

\section{The $C^*$-algebra $\cT^{(\lambda)}(L^\infty(\B^n)^{\T^m})$ and quasi-radial symbols}\label{sec:Tm_and_quasi-radial}
As in the previous section, we will consider a fixed partition $\kb \in \Z_+^m$ of $n$. This partition $\kb$ also fixes the groups $\U(\kb)$ and $\T^m$ defined as before.

We recall (see \cite{VasQuasiRadial}) that a symbol $a \in L^\infty(\B^n)$ is called $\kb$-quasi-radial if there is a measurable essentially bounded function $f$ such that $a(z) = f(|z_{(1)}|, \dots, |z_{(m)}|)$ for almost every $z \in \B^n$. In particular, it is clear that the space of $\kb$-quasi-radial essentially bounded symbols is precisely $L^\infty(\B^n)^{\U(\kb)}$. It was proved in \cite{VasQuasiRadial} that the $C^*$-algebra $\cT^{(\lambda)}(L^\infty(\B^n)^{\U(\kb)})$, generated by Toeplitz operators with $\kb$-quasi-radial symbols, is commutative. In this section, we will establish in this section some relations between the $C^*$-algebras $\cT^{(\lambda)}(L^\infty(\B^n)^{\U(\kb)})$ and $\cT^{(\lambda)}(L^\infty(\B^n)^{\T^m})$.

From Corollary~\ref{cor:HinvariantToeplitz} it follows that $\cT^{(\lambda)}(L^\infty(\B^n)^{\T^m}) \subset \End_{\T^m}(\cA^2_\lambda(\B^n))$, and so we can apply the map from Proposition~\ref{prop:intertwiningTm} followed by a projection to obtain the homomorphism of $C^*$-algebras given by
\begin{align*}
    \pi_\kapb : \cT^{(\lambda)}(L^\infty(\B^n)^{\T^m}) &\rightarrow \End(\cP_\kapb(\C^n))  \\
    T &\mapsto T|_{\cP_\kapb(\C^n)}
\end{align*}
for every $\kapb \in \N^m$.

\begin{corollary}\label{cor:pikapb_irreducible}
    For every $lambda > -1$, the homomorphism
    \[
        \pi_\kapb : \cT^{(\lambda)}(L^\infty(\B^n)^{\T^m}) \rightarrow \End(\cP_\kapb(\C^n))
    \]
    of $C^*$-algebras defines an irreducible representation of $\cT^{(\lambda)}(L^\infty(\B^n)^{\T^m})$ for every $\kapb \in \N^m$. In particular, if $m < n$, i.e.~$\kb \not= (1,\dots, 1) \in \N^n$, then $\cT^{(\lambda)}(L^\infty(\B^n)^{\T^m})$ is not commutative.
\end{corollary}
\begin{proof}
    By Proposition~\ref{prop:HToeplitz_density},
    $\cT^{(\lambda)}(L^\infty(\B^n)^{\T^m})$ is dense in $\End_{\T^m}(\cA^2_\lambda(\B^n))$ with respect to the strong operator topology. Using Proposition~\ref{prop:intertwiningTm}, we note that for every $\kapb \in \N^m$, the projection $\End_{\T^m}(\cA^2_\lambda(\B^n)) \rightarrow \End(\cP_\kapb(\C^n))$ given by $T \mapsto T|_{\cP_\kapb(\C^n)}$ is surjective and also continuous when both the domain and the target are endowed with the strong operator topology. It follows that, for every $\kapb \in \N^m$, the algebra $\pi_\kapb(\cT^{(\lambda)}(L^\infty(\B^n)^{\T^m}))$ is dense in $\End(\cP_\kapb(\C^n))$, in the strong operator topology, and so that $\pi_\kapb$ is surjective.

    If $\kb \not= (1, \dots, 1) \in \Z_+^n$, then there is some $j_0$ such that $k_{j_0} > 1$. It follows that for every $\kapb \in \N^m$ such that $\kappa_{j_0} \geq 1$, the space $\cP_{\kapb}(\C^n)$ has dimension at least $2$. Hence, for such $\kapb$, the $C^*$-algebra $\End(\cP_{\kapb}(\C^n))$ is not commutative and so $\cT^{(\lambda)}(L^\infty(\B^n)^{\T^m})$ is not commutative either by the first part.
\end{proof}

\begin{remark}\label{rmk:multfree_vs_commToeplitz}
    An alternative proof to the second claim of Corollary~\ref{cor:pikapb_irreducible} can be obtained from Propositions~\ref{prop:HToeplitz_density} and \ref{prop:isotypicTm}. We refer to Theorem~6.4 from \cite{DOQJFA} for further details.
\end{remark}

We now proceed to obtain a relationship between the Toeplitz operators belonging to the commutative $C^*$-algebra $\cT^{(\lambda)}(L^\infty(\B^n)^{\U(\kb)})$ and the (in general) non-commutative $C^*$-algebra $\cT^{(\lambda)}(L^\infty(\B^n)^{\T^m})$. We start by establishing some notation for the elements on both $C^*$-algebras. In the rest of this work, and following the notation of Proposition~\ref{prop:intertwiningU(kb)}, the identity map on $\cP_\kapb(\C^n)$ will be denoted by $I_\kapb$.

As a consequence of Proposition~\ref{prop:intertwiningU(kb)} and Corollary~\ref{cor:HinvariantToeplitz}, for every $\lambda > -1$ and for every $T \in \cT^{(\lambda)}(L^\infty(\B^n)^{\U(\kb)})$ there is a sequence $\gamma_{\lambda, T} \in \ell^\infty(\N^m)$ such that
\begin{equation}\label{eq:gamma_lambdaT}
    T = \bigoplus_{\kapb \in \N^m}
        \gamma_{\lambda, T}(\kapb) I_{\kapb},
        \quad
        (T \in \cT^{(\lambda)}(L^\infty(\B^n)^{\U(\kb)})),
\end{equation}
with respect to the isotypic decomposition from Proposition~\ref{prop:isotypicU(kb)}. In the special case where $T = T_a^{(\lambda)}$, for a symbol $a \in L^\infty(\B^n)^{\U(\kb)}$, we will write $\gamma_{\lambda, a} = \gamma_{\lambda, T_a^{(\lambda)}}$ so that we have
\begin{equation}\label{eq:gamma_lambdaa}
    T_a^{(\lambda)}
        = \bigoplus_{\kapb \in \N^m}
            \gamma_{\lambda, a}(\kapb) I_\kapb,
            \quad
            (a \in L^\infty(\B^n)^{\U(\kb)}).
\end{equation}
Similarly, by Proposition~\ref{prop:intertwiningTm} and Corollary~\ref{cor:HinvariantToeplitz}, for every $\lambda > -1$ and for every $T  \in \cT^{(\lambda)}(L^\infty(\B^n)^{\T^m})$ we have the decomposition
\begin{equation}\label{eq:block_lambdaT}
    T = \bigoplus_{\kapb \in \N^m}
            T|_{\cP_\kapb(\C^n)},
            \quad
            (T \in \cT^{(\lambda)}(L^\infty(\B^n)^{\T^m})),
\end{equation}
where $T|_{\cP_\kapb(\C^n)}$ can assume any value in $\End(\cP_\kapb(\C^n))$ (see Corollary~\ref{cor:pikapb_irreducible}). Correspondingly, for a symbol $a \in L^\infty(\B^n)^{\T^m}$ we now have
\begin{equation}\label{eq:block_lambdaa}
    T_a^{(\lambda)} = \bigoplus_{\kapb \in \N^m}
            T_a^{(\lambda)} |_{\cP_\kapb(\C^n)},
            \quad
            (a \in L^\infty(\B^n)^{\T^m}).
\end{equation}

Note that equations~\eqref{eq:gamma_lambdaT} and \eqref{eq:gamma_lambdaa} yield diagonal decompositions, corresponding to the commutativity of $\cT^{(\lambda)}(L^\infty(\B^n)^{\U(\kb)})$, while equations~\eqref{eq:block_lambdaT} and \eqref{eq:block_lambdaa} yield block diagonal decompositions, for the (in general) non-commutative $C^*$-algebra $\cT^{(\lambda)}(L^\infty(\B^n)^{\T^m})$. Hence, the next result provides us a way to relate the behavior of the block diagonal case, for the Toeplitz operators with $\T^m$-invariant symbols, in terms of the diagonal case, of the Toeplitz operators with $\kb$-quasi-radial symbols.

\begin{theorem}\label{thm:traceToeplitzTm}
    Let $\kb \in \Z_+^m$ be a fixed partition of $n$, and let $\T^m$ and $\U(\kb)$ be the corresponding subgroups of $\U(n)$ defined as before.
    Let $a \in L^\infty(\B^n)^{\T^m}$ be given and consider the symbol $\widehat{a} \in L^\infty(\B^n)^{\U(\kb)}$ defined by \eqref{eq:widehata} for $H = \U(\kb)$. Then, for every $\lambda > -1$ we have
    \[
        \tr(T_a^{(\lambda)}|_{\cP_\kapb(\C^n)}) =
            \tr(T_{\widehat{a}}^{(\lambda)} |_{\cP_\kapb(\C^n)}),
    \]
    for every $\kapb \in \N^m$. In particular, in the notation of the decompositions provided by \eqref{eq:gamma_lambdaa} and \eqref{eq:block_lambdaa} we have
    \[
        \gamma_{\lambda, \widehat{a}}(\kapb)
            = \frac{\tr(T_a^{(\lambda)} |_{\cP_\kapb(\C^n)})}{\dim \cP_\kapb(\C^n)},
    \]
    for every $\kapb \in \N^m$.
\end{theorem}
\begin{proof}
    From the decompositions given by \eqref{eq:gamma_lambdaa} and \eqref{eq:block_lambdaa}, the Toeplitz operators $T_a^{(\lambda)}$ and $T_{\widehat{a}}^{(\lambda)}$ preserve the finite dimensional subspace $\cP_\kapb(\C^n)$, for every $\kapb \in \N^m$. Let us fix $\kapb \in \N^m$ to prove the required identities.

    Choose $(f_j)_j$ any orthonormal basis for $\cP_\kapb(\C^n)$. Then, using Lemma~\ref{lem:widehatTa} as well as \eqref{eq:widehatT} we have
    \begin{align*}
        \tr(T_{\widehat{a}}^{(\lambda)}
            |_{\cP_\kapb(\C^n)})
            &= \tr(\widehat{T_a^{(\lambda)}}
            |_{\cP_\kapb(\C^n)}) =
            \sum_{j=1}^{\dim \cP_\kapb(\C^n)}
            \langle \widehat{T_a^{(\lambda)}} f_j,
                f_j \rangle_\lambda \\
            &= \int_{\U(\kb)}
                \sum_{j=1}^{\dim \cP_\kapb(\C^n)}
                \langle \pi_\lambda(A) T_a^{(\lambda)} |_{\cP_\kapb(\C^n)} \pi_\lambda(A)^{-1} f_j, f_j \rangle_\lambda \dif \mu(A) \\
            &= \int_{\U(\kb)} \tr(\pi_\lambda(A) |_{\cP_\kapb(\C^n)} T_a^{(\lambda)}|_{\cP_\kapb(\C^n)} \pi_\lambda(A)^{-1} |_{\cP_\kapb(\C^n)}) \dif \mu(A) \\
            &= \tr(T_a^{(\lambda)} |_{\cP_\kapb(\C^n)}),
    \end{align*}
    where we have used that $\mu$ is a probability measure. This proves the first identity. The second identity is an immediate consequence of the decompositions given by \eqref{eq:gamma_lambdaa} and \eqref{eq:block_lambdaa}.
\end{proof}

We can use the previous result to obtain a sort of spectral integral formula for the traces of the blocks of Toeplitz operators with $\T^m$-invariant symbols. This is the content of the next result. We will use from now the notation
\[
    \tau(\B^l) = \{r \in \R^l_+ \mid |r| < 1 \}
\]
for every $l \in \Z_+$.

\begin{theorem}\label{thm:traceToeplitzTm_integral}
    Let $\kb \in \Z_+^m$ be a fixed partition of $n$, and let $\T^m$ and $\U(\kb)$ be the corresponding subgroups of $\U(n)$ defined as before. For every $j =1, \dots, m$, choose $u_j \in \C^{k_j}$ unitary vectors. If $a \in L^\infty(\B^n)^{\T^m}$, then for every $\lambda > -1$ we have
    \begin{multline*}
        \tr(T_a^{(\lambda)}|_{\cP_{\kapb}(\C^n)}) =
          \frac{2^m\Gamma(n + \lambda + |\kapb| + 1)}{\Gamma(\lambda + 1) \prod_{j=1}^m \kappa_j !(k_j - 1)!} \times \\
        \times
        \int_{\U(\kb) \times \tau(\B^m)}
            a(r_1 A_1^{-1} u_1, \dots,
                r_m A_m^{-1} u_m)
                (1 - |r|^2)^\lambda  \\
        \times \dif \mu(A)
                \prod_{j=1}^m
                r_j^{2k_j + 2\kappa_j - 1} \dif r_j.
    \end{multline*}
    for every $\kapb \in \N^m$, where $\mu$ denotes the probability Haar measure of $\U(\kb)$.
\end{theorem}
\begin{proof}
    For a given symbol $a \in L^\infty(\B^n)^{\T^m}$, and using the notation of Theorem~\ref{thm:traceToeplitzTm}, let $\widehat{a} \in L^\infty(\B^n)^{\U(\kb)}$ be the corresponding $\kb$-quasi-radial symbol. Then, there exist an essentially bounded measurable function $f$ such that
    \[
        \widehat{a}(z)
            = f(|z_{(1)}|, \dots, |z_{(m)}|)
    \]
    for $z \in \B^n$. In particular, for every $r \in \tau(\B^m)$ we have
    \begin{align*}
        f(r_1, \dots, r_m)
            &= \widehat{a}(r_1 u_1,
                \dots, r_m u_m) \\
            &= \int_{\U(\kb)} a(r_1 A_1^{-1} u_1, \dots, r_m A_m^{-1} u_m) \dif \mu(A).
    \end{align*}
    For the $\kb$-quasi-radial symbol $\widehat{a}$, Lemma~3.1 from \cite{VasQuasiRadial} shows that
    \begin{align*}
        \gamma_{\lambda, \widehat{a}}(\kapb)
            &=  \frac{2^m \Gamma(n + \lambda + |\kapb| + 1)}{\Gamma(\lambda +1) \prod_{j=1}^m(k_j + \kappa_j - 1)!} \times \\
            &\times \int_{\tau(\B^m)}
                f(r_1, \dots, r_m)
                    (1 - |r|^2)^\lambda
                \prod_{j=1}^m
                    r_j^{2k_j + 2\kappa_j - 1}
                        \dif r_j,
    \end{align*}
    for every $\kapb \in \N^m$. We also have the following well known formula
    \[
        \dim \cP_{\kapb}(\C^n)
            = \prod_{j=1}^m
                \dim \cP_{\kappa_j}(\C^n)
            = \prod_{j=1}^m
                \frac{(k_j + \kappa_j -1)!}{\kappa_j! (k_j -1)!}.
    \]
    On the other hand, from  Theorem~\ref{thm:traceToeplitzTm} we have
    \[
        \tr(T_a^{(\lambda)}|_{\cP_{\kapb}(\C^n)}) =
            \dim \cP_{\kapb}(\C^n)
                \gamma_{\lambda, \widehat{a}}(\kapb)
    \]
    for every $\kapb \in \N^m$. Hence, the result is obtained by replacing into the last formula the expressions for $f(r_1, \dots, r_m)$, $\gamma_{\lambda, \widehat{a}}(\kapb)$ and $\dim \cP_\kapb(\C^n)$ obtained above.
\end{proof}

\begin{remark}\label{rmk:traceToeplitzTm_integral}
    We observe that the integral expression from Theorem~\ref{thm:traceToeplitzTm_integral} does not depend on the choice of the unitary vectors $u_j \in \C^{k_j}$. This can be seen as a consequence of the bi-invariance of the Haar measure $\mu$ of $\U(\kb)$ and the fact that the group $\U(k_j)$ acts transitively on the unitary vectors of $\C^{k_j}$, for every $j = 1, \dots, m$.
\end{remark}

As found in previous works (see, for example, \cite{BVQuasiEllipticI}, \cite{BHVRadial} and \cite{GMVRadial}) it is useful to consider the subspaces of sequences in $\ell^\infty$ spaces associated to diagonalizable Toeplitz operators. For this reason, we state the following immediate consequence of Theorem~\ref{thm:traceToeplitzTm}. Note that, even though the $C^*$-algebra $\cT^{(\lambda)}(L^\infty(\B^n)^{\T^m})$ is in general not commutative, the next result provides some information on the asymptotic behavior of the Toeplitz operators that belong to this~$C^*$-algebra.

\begin{corollary}
\label{cor:operatorsequences}
    With the notation from Theorem~\ref{thm:traceToeplitzTm} and of the decompositions provided by \eqref{eq:gamma_lambdaa} and \eqref{eq:block_lambdaa}, for every $\lambda > -1$, the following spaces of sequences are the same.
    \begin{enumerate}
        \item The space $\{ \gamma_{\lambda, a} \mid a \in L^\infty(\B^n)^{\U(\kb)} \}$.
        \item The space of sequences of the form
            \[
                \left(
                \frac{\tr(T_a^{(\lambda)} |_{\cP_\kapb(\C^n)})}{\dim \cP_\kapb(\C^n)}
                \right)_{\kapb \in \N^m},
            \]
            where $a \in L^\infty(\B^n)^{\T^m}$.
    \end{enumerate}
\end{corollary}

The case of radial Toeplitz operators ($\kb = (n)$ and $\U(\kb) = \U(n)$) is particularly interesting. This is because the results from \cite{BHVRadial} (see also \cite{GMVRadial}) provide very detailed information on the space of sequences obtained. Using \cite{BHVRadial} and Theorem~\ref{thm:traceToeplitzTm} we obtain the next result for Toeplitz operators with $\T$-invariant symbols ($m = 1$). We recall that the space of bounded slowly oscillating sequences is defined by
\[
    \mathrm{SO}(\N) =
        \{ x \in \ell^\infty(\N) \mid
            \lim_{\frac{r+1}{s+1}\to 1} |x_r - x_s| = 0
            \}.
\]
We refer to \cite{BHVRadial}, \cite{GMVRadial} and \cite{Schmidt} for further details and properties of this space.

\begin{theorem}\label{thm:traceToeplitzT_SON}
    Let $\T$ denote the center of the unitary group $\U(n)$. Let us consider, for every $\lambda > -1$, the space of sequences
    \[
        \mathcal{ST}^{(\lambda)}(\T,n) =
        \Bigg\{
            \Bigg(
                \frac{\kappa! (n-1)!}{(n + \kappa - 1)!}
                \tr(T_a^{(\lambda)}|_{\cP_\kappa(\C^n)})
            \Bigg)_{\kappa \in \N}
            \sVert[3] a \in L^\infty(\B^n)^\T
        \Bigg\}.
    \]
    Then, the $C^*$-algebra generated by the set $\mathcal{ST}^{(\lambda)}(\T,n)$ and the norm closure of $\mathcal{ST}^{(\lambda)}(\T,n)$ in $\ell^\infty(\N)$ are both equal to $\mathrm{SO}(\N)$.
\end{theorem}
\begin{proof}
    We observe that $\mathcal{ST}^{(\lambda)}(\T,n)$ is precisely the space of sequences described in Corollary~\ref{cor:operatorsequences}(2) for the case $m = 1$ and $\kb = (n)$. It follows from this corollary that $\mathcal{ST}^{(\lambda)}(\T,n)$ coincides with the space of sequences $\gamma_{\lambda,a}$ where $a \in L^\infty(\B^n)$ runs through all radial symbols. Hence, the result is a consequence of Theorem~5.4 and Corollary~5.5 from \cite{BHVRadial}.
\end{proof}

\begin{remark}
    The elements of $\mathcal{ST}^{(\lambda)}(\T,n)$ can be seen as sequences of normalized traces of the blocks from the decomposition \eqref{eq:block_lambdaa}, when $m = 1$, for Toeplitz operators $T_a^{(\lambda)}$ with essentially bounded $\T$-invariant symbols $a$. The normalization is given by the dimensions
    \[
        \dim \cP_\kappa(\C^n) =
            \frac{(n + \kappa - 1)!}{\kappa! (n - 1)!}
    \]
    of the subspaces with respect to which the decomposition \eqref{eq:block_lambdaa} is obtained. Note that for $n > 1$, the $C^*$-algebra $\cT^{(\lambda)}(L^\infty(\B^n)^\T)$ is non-commutative (see Corollary~\ref{cor:pikapb_irreducible}) and it is expected to be very complicated to describe. Nevertheless, Theorem~\ref{thm:traceToeplitzT_SON} provides asymptotic information for the Toeplitz operators that belong to $\cT^{(\lambda)}(L^\infty(\B^n)^\T)$.
\end{remark}

\section{Quasi-homogeneous symbols and commutative Banach algebras}\label{sec:quasi}
As before, we fix a partition $\kb \in \Z_+^m$ and the corresponding subgroups $\U(\kb)$ and $\T^m$ of $\U(n)$. We now further consider for every $j = 1, \dots, m$ the subgroup defined by
\[
    \U(\kb,j,\T) = \U(k_1)\times \dots \times U(k_{j-1}) \times \T I_{(j)} \times \U(k_{j+1}) \times \dots \times \U(k_m),
\]
where, as denoted before, $I_{(j)}$ is the identity map on $\C^{k_j}$. In particular, we have
\[
    \T^m \subset \U(\kb,j,\T) \subset \U(\kb)
\]
for every $j =1, \dots, m$. Hence, it will be convenient to assume, from now on, that $m \geq 2$ to avoid trivial cases. We also note that these inclusions yield the obvious implications for the invariance of symbols.

As proved in Corollary~\ref{cor:pikapb_irreducible}, for $m < n$, the $C^*$-algebra $\cT^{(\lambda)}(L^\infty(\B^n)^{\T^m})$ is non-commutative. However, we will use the subgroups $\U(\kb,j,\T)$ to build commutative Banach subalgebras of $\cT^{(\lambda)}(L^\infty(\B^n)^{\T^m})$ that enlarge the commutative $C^*$-algebra $\cT^{(\lambda)}(L^\infty(\B^n)^{\U(\kb)})$. This will be achieved by considering special symbols obtained from the group $\U(\kb,j,\T)$. We now introduce such symbols.

\begin{definition}\label{def:kbj-quasi}
    With the previous notation a symbol $a \in L^\infty(\B^n)$ is called $(\kb,j)$-quasi-radial quasi-homogeneous if it is $\U(\kb,j,\T)$-invariant. We will also say that $a$ is quasi-radial quasi-homogeneous for the group $\U(\kb,j,\T)$.
\end{definition}

Note that the space of essentially bounded $(\kb,j)$-quasi-radial quasi-homogeneous symbols is given by $L^\infty(\B^n)^{\U(\kb,j,\T)}$.

For our given partition $\kb$, we have considered the decomposition
\[
    z = (z_{(1)}, \dots, z_{(m)})
\]
defined above for every $z \in \C^n$. We will now recall further coordinate decompositions that will be useful for our purposes. More precisely, for a given $z \in \C^n$ we write
\[
    r_j = |z_{(j)}|, \quad
    z_{(j)} = r_j \xi_{(j)},
\]
for every $j =1, \dots, m$, where $\xi_{(j)} \in \mbS^{k_j}$ and $\mbS^{k_j}$ is the unit sphere in $\C^{k_j}$. We also consider the decomposition given by writing
\[
    \xi_{(j)} = t_{(j)}\cdot s_{(j)},
\]
for every $j = 1, \dots, m$, where $t_{(j)} \in \T^{k_j}$ and $s_{(j)} \in S_+^{k_j-1} = \mbS^{k_j} \cap \R_+^{k_j}$. Note that the expression $t_{(j)}\cdot s_{(j)}$ is the notation used for the $\T^{k_j}$-action on $\C^{k_j}$. Hence, we have
\[
    z_{(j)} = r_j t_{(j)}\cdot s_{(j)},
\]
for every $j = 1, \dots, m$. These yield corresponding expressions associated to every $z \in \C^n$ given by
\begin{align*}
    \xi &= (\xi_{(1)}, \dots, \xi_{(m)}) \in
        \mbS^{k_1} \times \dots \times \mbS^{k_m}, \\
    t &= (t_{(1)}, \dots, t_{(m)}) \in \T^n, \\
    s &= (s_{(1)}, \dots, s_{(m)}) \in
        S_+^{k_1 - 1} \times \dots \times S_+^{k_m - 1}.
\end{align*}
In particular, we have
\[
    z = (r_1 \xi_{(1)}, \dots, r_m \xi_{(m)})
        = (r_1 t_{(1)}\cdot s_{(1)}, \dots,
            r_m t_{(m)}\cdot s_{(m)}),
\]
for every $z \in \C^n$.

It is straightforward to prove the next characterization of the symbols introduced in Definition~\ref{def:kbj-quasi}.

\begin{lemma}\label{lem:kbj-quasi}
    With the previous notation, a symbol $a \in L^\infty(\B^n)$ is $(\kb,j)$-quasi-radial quasi-homogeneous if and only if the following equivalent conditions hold.
    \begin{enumerate}
        \item There is a function $f \in L^\infty(\tau(\B^m) \times \mbS^{k_j})$ that satisfies
            \[
                f(r, \eta\xi) = f(r,\xi)
            \]
            for every $r \in \tau(\B^m)$, $\xi \in \mbS^{k_j}$ and $\eta\in \T$, such that
            \[
                a(z) = f(|z_{(1)}|, \dots, |z_{(m)}|, \xi_{(j)}),
            \]
            for every $z \in \B^n$.
        \item There is a function $g \in L^\infty(\tau(\B^m) \times S_+^{k_j-1} \times \T^{k_j})$ that satisfies
            \[
                g(r,s,\eta t) = g(r,s,t)
            \]
            for every $r \in \tau(\B^m)$, $s \in S_+^{k_j-1}$, $t \in \T^{k_j}$ and $\eta \in \T$, such that
            \[
                a(z) = g(|z_{(1)}|, \dots, |z_{(m)}|, s_{(j)}, t_{(j)}),
            \]
            for every $z \in \B^n$.
    \end{enumerate}
\end{lemma}

The notions of quasi-homogeneous and pseudo-homogeneous symbols have been considered before (recall the remarks from the Introduction). We briefly describe those previous notions and compare them with Definition~\ref{def:kbj-quasi}. The claims below are immediate consequences of Lemma~\ref{lem:kbj-quasi}.

A quasi-homogeneous symbol is defined in \cite{VasQuasiRadial} as a function $\varphi_{p,q} \in L^\infty(\B^n)$ of the form
\[
    \varphi_{p,q}(z) = \xi^p \overline{\xi}^q
        = \prod_{j=1}^m
            \xi_{(j)}^{p_{(j)}}
            \overline{\xi}_{(j)}^{q_{(j)}},
\]
where $p, q \in \Z^n$. In \cite{VasQuasiRadial} the author considers further conditions that allow to obtain interesting properties for the corresponding Toeplitz operators. More precisely, the following condition is assumed to hold
\[
    |p_{(j)}| = |q_{(j)}|,
\]
for every $j = 1, \dots, m$. With this condition, the quasi-homogeneous symbol $\varphi_{p,q}$ is $\T^m$-invariant. Furthermore, for every $j = 1, \dots, m$, the symbol given~by
\[
    z \mapsto \xi_{(j)}^{p_{(j)}}
            \overline{\xi}_{(j)}^{q_{(j)}},
\]
is $(\kb,j)$-quasi-radial quasi-homogeneous. Hence, $\varphi_{p,q}$ is a product of quasi-radial quasi-homogeneous symbols from Definition~\ref{def:kbj-quasi} for varying values of~$j$.

On the other hand, in \cite{VasPseudoH} the author defines the $\kb$-pseudo-homogeneous symbols as essentially bounded functions of the form
\[
    \psi(z) = b(s_{(1)}, \dots, s_{(m)}) \prod_{j=1}^m t_{(j)}^{p_{(j)}},
\]
for every $z \in \B^n$, where $p \in \Z^n$. As before, further conditions are considered. More precisely, it is assumed that
\begin{align*}
    b(s_{(1)}, \dots, s_{(m)})
        &= \prod_{j=1}^m b_j(s_{(j)})  \\
     |p_{(j)}| &= 0, \quad
        \text{for all $j=1,\dots, m$},
\end{align*}
for some functions $b_j$. With these assumptions, the symbol $\psi$ is clearly $\T^m$-invariant. Furthermore, for every $j = 1,\dots, m$, the symbol given by
\[
    z \mapsto b_j(s_{(j)}) t_{(j)}^{p_{(j)}},
\]
is $(\kb,j)$-quasi-radial quasi-homogeneous. Again, with the previous assumptions, $\psi$ is a product of symbols from Definition~\ref{def:kbj-quasi}.

Finally, in \cite{VasTm2018} (see also \cite{Rodriguez2020}) it is considered the $\T^m$-invariant symbols of the form
\[
    g(z) = a(r_1, \dots, r_m) \prod_{j=h+1}^m b_j(s_{(j)}) c_j(t_{(j)})
\]
for every $z \in \B^n$. In this case, the functions $c_j$ are all assumed to satisfy the $\T$-invariance condition
\[
    c_j(\eta t_{(j)}) = c_j(t_{(j)}),
\]
for all $\eta \in \T$. With this assumption, it is now clear that, for every $j = 1, \dots, m$, the symbol given by
\[
    z \mapsto b_j(s_{(j)}) c_j(t_{(j)}),
\]
is $(\kb,j)$-quasi-radial quasi-homogeneous, and $g$ is a product of symbols from Definition~\ref{def:kbj-quasi}.

\begin{remark}\label{rmk:(k,j)-quasi}
    A consequence of the previous observations is that, for every $j = 1, \dots, m$, the quasi-radial quasi-homogeneous and pseudo-homogeneous symbols considered in the previous literature and that depend only on $z_{(j)}$ (for $z \in \B^n$) are particular cases of the more general $(\kb,j)$-quasi-radial quasi-homogeneous symbols from Definition~\ref{def:kbj-quasi}.
\end{remark}

Corollary~\ref{cor:HinvariantToeplitz} implies that every $(\kb,j)$-quasi-radial quasi-homogeneous symbol $a$ yields, for every $\lambda > -1$, a Toeplitz operator $T^{(\lambda)}_a$ that intertwines the representation $\pi_\lambda|_{\U(\kb,j,\T)}$. Hence, it will be useful to describe the operators intertwining such representation of the group $\U(\kb,j,\T)$.

\begin{proposition}\label{prop:EndU(kb,j,T)-tensorproduct}
    Let $\lambda > -1$, $\kb \in \Z_+^m$ and $j \in \{1, \dots, m\}$ be given. We will assume that $m \geq 2$. If $T \in \End_{\U(\kb,j,\T)}(\cA^2_\lambda(\B^n))$, then, for every $\kapb \in \N^m$ we have
    \[
        T(\cP_{\kapb}(\C^n)) \subset \cP_{\kapb}(\C^n),
    \]
    and there exists a linear map $M_{\lambda, T}(\kapb, j) \in \End(\cP_{\kappa_j}(\C^n))$ such that
    \[
        T|_{\cP_\kapb(\C^n)} =
            U_\kapb \circ (I_{(1)}\otimes \dots \otimes I_{(j-1)} \otimes M_{\lambda, T}(\kapb, j) \otimes I_{(j+1)} \otimes \dots \otimes I_{(m)}) \circ U_\kapb^{-1}
    \]
    where $U_\kapb : \bigotimes_{j=1}^m \cP_{\kappa_j}(\C^{k_j}) \rightarrow \cP_\kapb(\C^n)$ is the isomorphism given by Proposition~\ref{prop:Pkapb-tensorproduct}.
\end{proposition}
\begin{proof}
    For simplicity we will assume that $j = m$, since the general case can be considered similarly.

    Let $T \in \End_{\U(\kb,m,\T)}(\cA^2_\lambda(\B^n))$ be given. Since $\T^m \subset \U(\kb,m,\T)$ it follows that $T$ is $\T^m$-equivariant and by Proposition~\ref{prop:intertwiningTm} it follows that $T(\cP_\kapb(\C^n)) \subset \cP_\kapb(\C^n)$ for every $\kapb \in \N^m$. Hence, we can consider the linear map $\widetilde{T} = U_\kapb^{-1}\circ T \circ U_\kapb$ acting on the tensor product $\bigotimes_{j=1}^m \cP_{\kappa_j}(\C^{k_j})$. By Proposition~\ref{prop:Pkapb-tensorproduct}, $U_\kapb$ is $\U(\kb)$-equivariant and so it is $\U(\kb,m,\T)$-equivariant as well. It follows that $\widetilde{T}$ is an intertwining map for the $\U(\kb,m,\T)$-action given by the outer tensor product action on $\bigotimes_{j=1}^m \cP_{\kappa_j}(\C^{k_j})$.

    We now fix $\kapb \in \N^m$. We will consider the partition $\kb' = (k_1, \dots, k_{m-1})$ of $n' = |\kb'|$. Choose $(f_l)_{l = 1}^N$ a basis for $\cP_{\kappa_m}(\C^{k_m})$. Then, we have a direct sum decomposition
    \begin{equation}\label{eq:sumkapbfl}
        \bigotimes_{j=1}^m \cP_{\kappa_j}(\C^{k_j})
        = \bigoplus_{l = 1}^N
        \bigotimes_{j=1}^{m-1} \cP_{\kappa_j}(\C^{k_j})
        \otimes f_l.
    \end{equation}
    We observe that the action of the subgroup $\U(\kb')$ of $\U(\kb)$ on the terms of \eqref{eq:sumkapbfl} is given by
    \[
        A \cdot (p \otimes f_l) =
            (A \cdot p)\otimes f_l
    \]
    for every $A \in \U(\kb')$, $p \in \bigotimes_{j=1}^{m-1} \cP_{\kappa_j}(\C^{k_j})$ and $l = 1, \dots, N$. It follows that the direct sum \eqref{eq:sumkapbfl} is a decomposition into $\U(\kb')$-submodules, which are irreducible by Proposition~\ref{prop:Pkapb-tensorproduct}. In fact, these $\U(\kb')$-submodules are all isomorphic to the irreducible $\U(\kb')$-module $\bigotimes_{j=1}^{m-1} \cP_{\kappa_j}(\C^{k_j})$.

    Hence, \eqref{eq:sumkapbfl} allows us to write $\widetilde{T}$ as a linear transformation
    \[
        \widetilde{T} : \bigoplus_{l = 1}^N \bigotimes_{j=1}^{m-1} \cP_{\kappa_j}(\C^{k_j}) \otimes f_l \rightarrow \bigoplus_{l = 1}^N \bigotimes_{j=1}^{m-1} \cP_{\kappa_j}(\C^{k_j}) \otimes f_l
    \]
    which we know that is intertwining for $\U(\kb,m,\T)$ and so it is intertwining for the subgroup $\U(\kb')$ as well. If we denote by $\pi_l$ the projection into the $l$-th term of the direct sum \eqref{eq:sumkapbfl}, it follows from Schur's Lemma that for every $l_1, l_2 = 1, \dots, N$, there is a constant $c_{l_1 l_2} \in \C$ such that
    \[
        \pi_{l_2} \circ \widetilde{T}|_{\bigotimes_{j=1}^{m-1} \cP_{\kappa_j}(\C^{k_j}) \otimes f_{l_1}} : \bigotimes_{j=1}^{m-1} \cP_{\kappa_j}(\C^{k_j}) \otimes f_{l_1} \rightarrow \bigotimes_{j=1}^{m-1} \cP_{\kappa_j}(\C^{k_j}) \otimes f_{l_2}
    \]
    is given by
    \[
        p \otimes f_{l_1} \mapsto c_{l_2 l_1} p \otimes f_{l_2},
    \]
    for every $p \in \bigotimes_{j=1}^{m-1} \cP_{\kappa_j}(\C^{k_j})$.

    Let $p \in \bigotimes_{j=1}^{m-1} \cP_{\kappa_j}(\C^{k_j})$ and $f \in \cP_{\kappa_m}(\C^{k_m})$ be given, and let us write
    \[
        f = \sum_{l=1}^N a_l f_l
    \]
    for some $a_l \in \C$. Then, we obtain from the previous remarks the following
    \begin{align*}
        \widetilde{T}(p \otimes f)
            &= \sum_{l_1=1}^N a_{l_1} \widetilde{T}(p\otimes f_{l_1}) \\
            &= \sum_{l_1, l_2 =1}^N
                a_{l_1} c_{l_2 l_1} p \otimes f_{l_2}  \\
            &= p \otimes
                \left( \sum_{l_1, l_2 =1}^N
                c_{l_2 l_1} a_{l_1} f_{l_2}
                \right).
    \end{align*}
    We conclude that
    \[
        \widetilde{T} =
        I_{(1)} \otimes \dots \otimes I_{(m-1)} \otimes M_{\lambda, T}(\kapb,m),
    \]
    where $M_{\lambda, T}(\kapb,m)$ is the element of $\End_{\kappa_m}(\C^{k_m})$ whose matrix is given by the coefficients $c_{l_1 l_2}$ with respect to the basic elements $f_l$ ($l = 1, \dots, N$).
\end{proof}

\begin{remark}
    The expression for the operators $T \in \End_{\U(\kb,j,\T)}(\cA^2_\lambda(\B^n))$ as a tensor product, up to the isomorphism $U_\kapb$, is similar to one obtained in Theorem~5.2 from \cite{VasTm2018}. However, the expression from \cite{VasTm2018} is given only for Toeplitz operators, acting on Fock spaces, with symbols of the form
    \[
        z \mapsto b_j(s_{(j)}) c_j(t_{(j)})
    \]
    defined on $\C^n$, where $c_j$ is $\T$-invariant. Meanwhile, the expression obtained in our Proposition~\ref{prop:EndU(kb,j,T)-tensorproduct} holds for the more general $\U(\kb,j,\T)$-intertwining operators acting on Bergman spaces. In particular, it also holds for Toeplitz operators, acting on Bergman spaces, whose symbols are our more general $(\kb,j)$-quasi-radial quasi-homogeneous symbols as given in Definition~\ref{def:kbj-quasi}.
\end{remark}

Proposition~\ref{prop:EndU(kb,j,T)-tensorproduct} allows to single out the following commuting Toeplitz operators with special invariant symbols.

\begin{theorem}\label{thm:quasisymbols_commToeplitz}
    Let $\kb \in \Z_+^m$ be a fixed partition of $n$ with $m \geq 2$ and let $j_1, j_2 \in \{1, \dots, m\}$ be two different indices. Let $a, b \in L^\infty(\B^n)$ be given. If $a$ and $b$ are quasi-radial quasi-homogeneous symbols for the subgroups $\U(\kb,j_1,\T)$ and $\U(\kb,j_2,\T)$, respectively, then we have
    \[
        T^{(\lambda)}_a T^{(\lambda)}_b =
            T^{(\lambda)}_b T^{(\lambda)}_a
    \]
    for every $\lambda > -1$.
\end{theorem}
\begin{proof}
    By Corollary~\ref{cor:HinvariantToeplitz} we have for every $\lambda > -1$
    \[
        T^{(\lambda)}_a \in
            \End_{\U(\kb,j_1,\T)}(\cA^2_\lambda(\B^n)), \quad
        T^{(\lambda)}_b \in
            \End_{\U(\kb,j_2,\T)}(\cA^2_\lambda(\B^n)).
    \]
    It follows from Proposition~\ref{prop:EndU(kb,j,T)-tensorproduct} that both operators preserve the subspaces $\cP_{\kapb}(\C^n)$, for every $\kapb \in \N^m$, whose Hilbert direct sum is the whole Bergman space $\cA^2_\lambda(\B^n)$. Furthermore, in the notation of Proposition~\ref{prop:EndU(kb,j,T)-tensorproduct} we also have
    \begin{align*}
        T^{(\lambda)}_a&|_{\cP_\kapb(\C^n)} = \\
            =&\;
            U_\kapb \circ (I_{(1)}\otimes \dots \otimes I_{(j_1-1)} \otimes M_{\lambda, T^{(\lambda)}_a}(\kapb, j_1) \otimes I_{(j_1+1)} \otimes \dots \otimes I_{(m)}) \circ U_\kapb^{-1} \\
        T^{(\lambda)}_b&|_{\cP_\kapb(\C^n)} = \\
            =&\;
            U_\kapb \circ (I_{(1)}\otimes \dots \otimes I_{(j_2-1)} \otimes M_{\lambda, T^{(\lambda)}_b}(\kapb, j_2) \otimes I_{(j_2+1)} \otimes \dots \otimes I_{(m)}) \circ U_\kapb^{-1}.
    \end{align*}
    The result now follows since $j_1 \not= j_2$.
\end{proof}

The previous result allows to obtain commutative Banach algebras generated by Toeplitz operators with quasi-radial quasi-homogeneous symbols properly chosen.

\begin{corollary}\label{cor:quasi_commBanach}
    Let $\kb \in \Z_+^m$ be a fixed partition of $n$ with $m \geq 2$. For every $j = 1, \dots, m$, let $a_j$ be a $(\kb,j)$-quasi-radial quasi-homogeneous essentially bounded symbol. Let us consider the set of symbols
    \[
        \cS = L^\infty(\B^n)^{\U(\kb)} \cup \{ a_j \mid j =1, \dots, m\}.
    \]
    Then, for every $\lambda > -1$, the Banach algebra $\cT^{(\lambda)}(\cS)$ generated by the Toeplitz operators with symbols in $\cS$ is commutative.
\end{corollary}
\begin{proof}
    Let us fix $\lambda > -1$. Corollary~\ref{cor:HinvariantToeplitz} yields
    \begin{align*}
        T^{(\lambda)}_a \in
            \End_{\U(\kb)}(\cA^2_\lambda(\B^n)),
            & \text{ for all } a \in L^\infty(\B^n)^{\U(\kb)} \\
        T^{(\lambda)}_{a_j} \in
            \End_{\T^m}(\cA^2_\lambda(\B^n)),
            & \text{ for all } j = 1, \dots, m,
    \end{align*}
    which, using Corollary~\ref{cor:centerEndTm}, implies that
    \[
        [T^{(\lambda)}_a,T^{(\lambda)}_{a_j}] = 0
    \]
    for every $a \in L^\infty(\B^n)^{\U(\kb)}$ and $j = 1, \dots, m$. It is also well known, as observed before, that
    \[
        [T^{(\lambda)}_a,T^{(\lambda)}_b] = 0
    \]
    for every $a,b \in L^\infty(\B^n)^{\U(\kb)}$.
    On the other hand, Theorem~\ref{thm:quasisymbols_commToeplitz} shows that
    \[
        [T^{(\lambda)}_{a_{j_1}}, T^{(\lambda)}_{a_{j_2}}] = 0
    \]
    for every $j_1, j_2 = 1, \dots, m$. And so the result follows.
\end{proof}

\begin{remark}\label{rmk:BanachnotCstar}
    Note that the Banach algebras described in Corollary~\ref{cor:quasi_commBanach} are not necessarily $C^*$-algebras. The main reason is that the operator $T^{(\lambda)}_{a_j}$ is in general not normal unless, for example, $a_j$ is real-valued.
\end{remark}

\section{Toeplitz operators with quasi-homogeneous symbols acting on monomials}\label{sec:acting_on_monomials}
We now consider Toeplitz operators with quasi-radial quasi-homogeneous symbols as given in Definition~\ref{def:kbj-quasi} and determine their actions on the monomial basic elements. We provide a complete description of the blocks in the block diagonal decomposition obtained from Proposition~\ref{prop:EndU(kb,j,T)-tensorproduct}. Such description is given in terms of integral formulas involving the corresponding symbols.

We start by introducing some additional notation on the use of our coordinates. Following Proposition~\ref{prop:EndU(kb,j,T)-tensorproduct}, besides our partition $\kb \in \N^m$, with $m \geq 2$, we fix an integer $j \in \{1, \dots, m\}$. For every $z \in \C^n$, we will denote
\[
    z_{\widehat{(j)}} =
    (z_{(1)}, \dots, z_{(j-1)}, z_{(j+1)}, \dots, z_{(m)}) \in \C^{n-k_j}.
\]
In particular, for every $\alpha \in \N^n$ we write
\[
    \alpha_{\widehat{(j)}} =
    (\alpha_{(1)}, \dots, \alpha_{(j-1)}, \alpha_{(j+1)}, \dots, \alpha_{(m)}) \in \N^{n-k_j}.
\]
With this notation we can write
\[
    z^\alpha = {z_{\widehat{(j)}}}^{\alpha_{\widehat{(j)}}} {z_{(j)}}^{\alpha_{(j)}}
\]
for every $z \in \C^n$ and $\alpha \in \N^n$.

On the other hand, using the partition $\kb$ of $n$ we define the partition $\kb_{\widehat{j}} = (k_1, \dots, k_{j-1}, k_{j+1}, \dots, k_m) \in \Z_+^{m-1}$ of $n-k_j$. This yields corresponding decompositions
\[
    w =
    (w_{(1)}, \dots, w_{(j-1)}, w_{(j+1)}, \dots, w_{(m)}),
\]
for $w \in \C^{n-k_j}$, and
\[
    \beta =
    (\beta_{(1)}, \dots, \beta_{(j-1)}, \beta_{(j+1)}, \dots, \beta_{(m)}),
\]
for $\beta \in \N^{n-k_j}$. Note that in both cases we are enumerating so that we omit the component corresponding to $(j)$. This notation will be convenient for our formulas below. In fact, we relate the decompositions obtained from $\kb$ and $\kb_{\widehat{j}}$ by fixing the embedding
\begin{align*}
    \C^{n-k_j} = \prod_{l\not=j} \C^{k_l} &\hookrightarrow \C^n  \\
    (w_{(1)}, \dots, w_{(j-1)}, w_{(j+1)}, \dots, w_{(m)}) &\mapsto (w_{(1)}, \dots, w_{(j-1)}, 0, w_{(j+1)}, \dots, w_{(m)}).
\end{align*}

With the notation just introduced, it follows that we have for every $\kapb \in \N^m$ the Hilbert direct sum decomposition
\begin{equation}\label{eq:cPkapb_sumcPkappaj}
    \cP_{\kapb}(\C^n) =
        \bigoplus_{\substack{\alpha \in \N^{n-k_j} \\ |\alpha_{(l)}| = \kappa_l, l \not= j}}
        {z_{\widehat{(j)}}}^\alpha \cP_{\kappa_j}(\C^{k_j}).
\end{equation}
Note that in this expression, the subspace ${z_{\widehat{(j)}}}^{\alpha} \cP_{\kappa_j}(\C^{k_j})$ is the linear span of the monomials of the form ${z_{\widehat{(j)}}}^{\alpha} {z_{(j)}}^{\alpha'}$ with $\alpha \in \N^{n-k_j}$ fixed and $\alpha'$ running through all the elements of $\N^{k_j}$ with length $\kappa_j$. Also, there is a natural isomorphism
\begin{align*}
    {z_{\widehat{(j)}}}^\alpha \cP_{\kappa_j}(\C^{k_j}) &\rightarrow
    {z_{\widehat{(j)}}}^\beta \cP_{\kappa_j}(\C^{k_j}) \\
    {z_{\widehat{(j)}}}^\alpha p(z_{(j)}) &\mapsto  {z_{\widehat{(j)}}}^\beta p(z_{(j)})
\end{align*}
for every $\alpha, \beta \in \N^{n-k_j}$ such that $|\alpha_{(l)}| = |\beta_{(l)}| = \kappa_l$, for $l \not= j$, and $p \in \cP_{\kappa_j}(\C^{k_j})$. We will use this isomorphism in the rest of this work. In particular, all the terms in \eqref{eq:cPkapb_sumcPkappaj} are naturally isomorphic as vector spaces.

Recall that, by Proposition~\ref{prop:EndU(kb,j,T)-tensorproduct}, any operator in $\End_{\U(\kb,j,\T)}(\cA^2_\lambda(\B^n))$ leaves invariant the subspace $\cP_\kapb(\C^n)$, for every $\kapb \in \N^m$. We now obtain some additional properties for the operators that intertwine the action of~$\U(\kb,j,\T)$.

\begin{proposition}\label{prop:EndU(kb,j,T)_matrixcoef}
    Let $\kb \in \Z_+^m$ be a fixed partition of $n$ with $m \geq 2$ and let $j \in \{ 1, \dots, m\}$ be given. If $T \in \End_{\U(\kb,j,\T)}(\cA^2_\lambda(\B^n))$, then, for every $\kapb \in \N^m$, the operator $T$ preserves the decomposition \eqref{eq:cPkapb_sumcPkappaj}. Furthermore, the corresponding induced linear maps acting on the terms of \eqref{eq:cPkapb_sumcPkappaj} do not depend on $\alpha$. More precisely, for every $\kapb \in \N^m$ the following properties hold.
    \begin{enumerate}
        \item For $\alpha, \beta \in \N^n$ such that $|\alpha_{(l)}| = |\beta_{(l)}| = \kappa_l$, for all $l = 1, \dots, m$, and satisfying $\alpha_{\widehat{(j)}} \not= \beta_{\widehat{(j)}}$ we have
            \[
                \langle T(z^\alpha), z^\beta
                \rangle_\lambda = 0,
            \]
            and so $T$ preserves the terms of the decomposition \eqref{eq:cPkapb_sumcPkappaj}.
        \item For $\alpha, \beta \in \N^{n-k_j}$ such that $|\alpha_{(l)}| = |\beta_{(l)}| = \kappa_l$, for all $l \not= j$, the diagram
            \[
            \xymatrix{
                {z_{\widehat{(j)}}}^{\alpha} \cP_{\kappa_j}(\C^{k_j}) \ar[r]^T \ar[d] &
                {z_{\widehat{(j)}}}^{\alpha} \cP_{\kappa_j}(\C^{k_j}) \ar[d] \\
                {z_{\widehat{(j)}}}^{\beta} \cP_{\kappa_j}(\C^{k_j}) \ar[r]^T &
                {z_{\widehat{(j)}}}^{\beta} \cP_{\kappa_j}(\C^{k_j})
            }
            \]
            is commutative, where the vertical arrows are given by the natural isomorphism described above.
    \end{enumerate}
\end{proposition}
\begin{proof}
    Let us fix $\kapb \in \N^m$. Note that Proposition~\ref{prop:EndU(kb,j,T)-tensorproduct} implies that the operator $T$ preserves~$\cP_\kapb(\C^n)$.

    We observe that, with respect to the isomorphism $U_\kapb$ from Proposition~\ref{prop:Pkapb-tensorproduct}, the direct sum decomposition from \eqref{eq:cPkapb_sumcPkappaj} corresponds to the decomposition
    \begin{multline*}
        \bigotimes_{l=1}^m \cP_{\kappa_l}(\C^{k_l}) = \\
        = \bigoplus_{\substack{\alpha \in \N^{n-k_j} \\ |\alpha_{(l)}| = \kappa_l \\ l \not= j}}
        {z_{(1)}}^{\alpha_{(1)}} \otimes \dots \otimes {z_{(j-1)}}^{\alpha_{(j-1)}} \otimes \cP_{\kappa_j}(\C^{k_j}) \otimes {z_{(j+1)}}^{\alpha_{(j+1)}} \otimes \dots \otimes {z_{(m)}}^{\alpha_{(m)}}.
    \end{multline*}
    Hence, Proposition~\ref{prop:EndU(kb,j,T)-tensorproduct} implies that $T|_{\cP_\kapb(\C^n)}$ preserves the decomposition \eqref{eq:cPkapb_sumcPkappaj}. This proves (1) by the orthogonality of the monomial basis. Furthermore, the expression for $T|_{\cP_\kapb(\C^n)}$ in Proposition~\ref{prop:EndU(kb,j,T)-tensorproduct} and the correspondence of the last direct sum with \eqref{eq:cPkapb_sumcPkappaj} also prove that the diagram from (2) is indeed commutative.
\end{proof}

\begin{remark}\label{rmk:Tblockdiagonalconst}
    Let us assume the notation from Proposition~\ref{prop:EndU(kb,j,T)_matrixcoef} and its proof, and let us choose $T \in \End_{\U(\kb,j,\T)}(\cA^2_\lambda(\B^n))$. Conclusions (1) and (2) of this proposition can be explained pictorially as follows. For every $\kapb \in \N^m$, there is a linear transformation $T_{\kapb,j} \in \End(\cP_{\kappa_j}(\C^{k_j}))$ such that $T|_{\cP_\kapb(\C^n)}$ can be written as
    \begin{equation}\label{eq:Tconstantmatrix}
    \begin{array}{ccccc}
            && {z_{\widehat{(j)}}}^{\alpha_1} \cP_{\kappa_j}(\C^{k_j}) & \cdots & {z_{\widehat{(j)}}}^{\alpha_N} \cP_{\kappa_j}(\C^{k_j}) \\
            &{z_{\widehat{(j)}}}^{\alpha_1} \cP_{\kappa_j}(\C^{k_j}) & T_{\kapb,j}& \cdots & 0 \\
            T|_{\cP_\kapb(\C^n)} :& \vdots &\vdots & \ddots & \vdots \\
            &{z_{\widehat{(j)}}}^{\alpha_N} \cP_{\kappa_j}(\C^{k_j}) & 0 & \cdots & T_{\kapb,j}
    \end{array}
    \end{equation}
    where $\alpha_1, \dots, \alpha_N$ is some enumeration of the multi-indices $\alpha \in \N^{n-k_j}$ such that $|\alpha_{(l)}| = \kappa_l$ for all $l \not= j$. We have also considered the natural isomorphism between ${z_{\widehat{(j)}}}^{\alpha} \cP_{\kappa_j}(\C^{k_j})$ and $\cP_{\kappa_j}(\C^{k_j})$, for every such $\alpha$. Hence, the block diagonal structure of $T$ coming from the fact that it preserves the decomposition
    \[
        \cA^2_\lambda(\B^n) =
            \bigoplus_{\kapb \in \N^m}
                \cP_\kapb(\C^n)
    \]
    is even finer, since the block corresponding to each restriction $T|_{\cP_\kapb(\C^n)}$ is itself a block diagonal operator constant along the diagonal via the linear transformation $T_{\kapb,j} \in \End(\cP_{\kappa_j}(\C^{k_j}))$.
\end{remark}

As noted in Remark~\ref{rmk:Tblockdiagonalconst}, many of the inner products $\langle T(z^\alpha), z^\beta \rangle_\lambda$ vanish when $T \in \End_{\U(\kb,j,\T)}(\cA^2_\lambda(\B^n))$ is restricted to the subspaces $\cP_\kapb(\C^n)$. The next results compute the rest of the inner products when $T = T_a^{(\lambda)}$ for $a \in L^\infty(\B^n)^{\U(\kb,j,\T)}$ in terms of the symbol $a$. We also provide a description of the matrix coefficients for the linear transformations $(T_a^{(\lambda)})_{\kapb,j}$, in the notation of Remark~\ref{rmk:Tblockdiagonalconst}. In what follows, we will denote by $\dif \xi_{(l)}$ the measure on $\mbS^{k_l}$ obtained from the Riemannian structure inherited from $\C^{k_l}$.

\begin{theorem}\label{thm:Tkappaj_withf_forToeplitz}
    Let $\kb \in \Z_+^m$ be a fixed partition of $n$ with $m \geq 2$ and let $j \in \{1, \dots, m\}$ be given. If $a \in L^\infty(\B^n)$ is a $(\kb,j)$-quasi-radial quasi-homogeneous symbol and $f$ is the corresponding function considered in Lemma~\ref{lem:kbj-quasi}(1), then for every $\lambda > -1$ and $\kapb \in \N^m$ we have
    \begin{multline*}
        \langle T_a^{(\lambda)}(z^\alpha), z^\beta \rangle_\lambda =
        \frac{2^{m-1}\Gamma(n + \lambda + 1) \alpha_{\widehat{(j)}}!}{\pi^{k_j} \Gamma(\lambda + 1) \prod_{l\not=j}(k_l + \kappa_l - 1)!} \times \\
        \times
        \int_{\tau(\B^m)\times \mbS^{k_j}}
            f(r,\xi_{(j)}) {\xi_{(j)}}^{\alpha_{(j)}} {\overline{\xi}_{(j)}}^{\beta_{(j)}} (1 - |r|^2)^\lambda \prod_{l=1}^m r_l^{2k_l+2\kappa_l-1} \dif r_l \dif \xi_{(j)},
    \end{multline*}
    for every $\alpha, \beta \in \N^n$ such that $|\alpha_{(l)}| = |\beta_{(l)}| = \kappa_l$, for all $l = 1, \dots, m$, that satisfy $\alpha_{\widehat{(j)}} = \beta_{\widehat{(j)}}$.
\end{theorem}
\begin{proof}
    Let us fix $\kapb$, $\alpha$, $\beta$ as in the statement. We compute
    \begin{align*}
        \langle T_a^{(\lambda)} & (z^\alpha), z^\beta \rangle_\lambda =
        \langle a z^\alpha, z^\beta \rangle_\lambda =  \\
        =&\; c_\lambda \int_{\B^n} a(z) |{z_{\widehat{(j)}}}^{\alpha_{\widehat{(j)}}}|^2 {z_{(j)}}^{\alpha_{(j)}} {\overline{z}_{(j)}}^{\beta_{(j)}} (1 - |z|^2)^\lambda \dif v(z),
    \end{align*}
    introducing spherical coordinates on each of the factors of $\C^n = \prod_{l=1}^m \C^{k_l}$ we have
    \begin{align*}
        =&\; c_\lambda \int_{\tau(\B^m) \times \prod_{l=1}^m \mbS^{k_l}} f(r, \xi_{(j)}) \prod_{l=1}^m r_l^{2\kappa_l} \prod_{l\not=j} |{\xi_{(l)}}^{\alpha_{(l)}}|^2 {\xi_{(j)}}^{\alpha_{(j)}} {\overline{\xi}_{(j)}}^{\beta_{(j)}} \times \\
        &\times (1 - |r|^2)^\lambda \prod_{l=1}^m r_l^{2k_l-1} \dif r_l \dif \xi_{(l)} \\
        =&\; \frac{\Gamma(n+\lambda+1)}{\pi^n \Gamma(\lambda+1)} \prod_{l\not=j}  \int_{\mbS^{k_l}} |{\xi_{(l)}}^{\alpha_{(l)}}|^2 \dif \xi_{(l)} \times \\
        &\times
        \int_{\tau(\B^m)\times \mbS^{k_j}}
            f(r,\xi_{(j)}) {\xi_{(j)}}^{\alpha_{(j)}} {\overline{\xi}_{(j)}}^{\beta_{(j)}} (1 - |r|^2)^\lambda \prod_{l=1}^m r_l^{2k_l+2\kappa_l-1} \dif r_l \dif \xi_{(j)},
    \end{align*}
    and the result follows by substituting the known values for the integrals over $\mbS^{k_l}$ for $l \not= j$.
\end{proof}

Recall that the canonical orthonormal monomial basis of $\cA^2_\lambda(\B^n)$ is given by
\[
    e_\alpha(z) = \sqrt{\frac{\Gamma(n+\lambda+|\alpha|+1)}{\alpha! \Gamma(n+\lambda+1)}} z^\alpha,
\]
where $\alpha \in \N^n$. The next result is a consequence of these expressions and of Theorem~\ref{thm:Tkappaj_withf_forToeplitz}.

\begin{corollary}\label{cor:Tkappaj_withf_forToeplitz}
    With the notation of Theorem~\ref{thm:Tkappaj_withf_forToeplitz} and Remark~\ref{rmk:Tblockdiagonalconst}, if $a \in L^\infty(\B^n)$ is a $(\kb,j)$-quasi-radial quasi-homogeneous symbol, then, for every $\kapb \in \N^m$, the coefficients of the linear map $(T_a^{(\lambda)})_{\kapb,j}$ are given by
    \begin{multline*}
        \langle T_a^{(\lambda)}(e_\alpha), e_\beta \rangle_\lambda =
        \frac{2^{m-1} \Gamma(n+\lambda+|\kapb|+1)}{\pi^{k_j} \Gamma(\lambda+1) \sqrt{\alpha_{(j)}! \beta_{(j)}!} \prod_{l\not=j} (k_l+\kappa_l-1)!} \times \\
        \times
        \int_{\tau(\B^m)\times \mbS^{k_j}}
            f(r,\xi_{(j)}) {\xi_{(j)}}^{\alpha_{(j)}} {\overline{\xi}_{(j)}}^{\beta_{(j)}} (1 - |r|^2)^\lambda \prod_{l=1}^m r_l^{2k_l+2\kappa_l-1} \dif r_l \dif \xi_{(j)},
    \end{multline*}
    for every $\alpha, \beta \in \N^n$ such that $|\alpha_{(l)}| = |\beta_{(l)}| = \kappa_l$, for all $l = 1, \dots, m$, that satisfy $\alpha_{\widehat{(j)}} = \beta_{\widehat{(j)}}$.
\end{corollary}

\begin{remark}
    Note that the formula obtained in Corollary~\ref{cor:Tkappaj_withf_forToeplitz} shows indeed that, with respect to an orthonormal basis, the matrix of the restriction of $T_a^{(\lambda)}$ to ${z_{\widehat{(j)}}}^{\alpha} \cP_{\kappa_j}(\C^n)$ does not depend on $\alpha \in \N^{n-k_j}$ satisfying $|\alpha_{(l)}| = \kappa_l$, for every $l \not= j$. This is in accordance with Proposition~\ref{prop:EndU(kb,j,T)_matrixcoef}(2). It also corresponds to the expression  \eqref{eq:Tconstantmatrix} applied to $T = T^{(\lambda)}_a$.
\end{remark}

We now consider the matrix coefficients for the Toeplitz operators with symbols in $L^\infty(\B^n)^{\U(\kb,j,\T)}$ corresponding to their representation in terms of the coordinates $(s,t)$. We will denote by $\dif t_{(j)}$ the Haar measure on $\T^{k_j}$ with total mass $(2\pi)^{k_j}$ and with $\dif s_{(j)}$ the measure on $S^{k_j-1}_+$ obtained from the Riemannian structure inherited from $\R^{k_j}$. In particular, we have the measure change of coordinates
\[
    \dif \xi_{(j)} = s_{(j)}^{1_{k_j}} \dif s_{(j)} \dif t_{(j)}
\]
where we denote $1_{k_j} = (1, \dots, 1) \in \N^{k_j}$. We obtain the next results corresponding to Theorem~\ref{thm:Tkappaj_withf_forToeplitz} and Corollary~\ref{cor:Tkappaj_withf_forToeplitz}. They are obtained by a change of coordinates using the previous expression for $\dif \xi_{(j)}$.

\begin{theorem}\label{thm:Tkappaj_withg_forToeplitz}
    Let $\kb \in \Z_+^m$ be a fixed partition of $n$ with $m \geq 2$ and let $j \in \{1, \dots, m\}$ be given. If $a \in L^\infty(\B^n)$ is a $(\kb,j)$-quasi-radial quasi-homogeneous symbol and $g$ is the corresponding function considered in Lemma~\ref{lem:kbj-quasi}(2), then for every $\lambda > -1$ and $\kapb \in \N^m$ we have
    \begin{multline*}
        \langle T_a^{(\lambda)}(z^\alpha), z^\beta \rangle_\lambda =
        \frac{2^{m-1}\Gamma(n + \lambda + 1) \alpha_{\widehat{(j)}}!}{\pi^{k_j} \Gamma(\lambda + 1) \prod_{l\not=j}(k_l + \kappa_l - 1)!} \times \\
        \times
        \int_{\tau(\B^m)\times S_+^{k_j-1} \times \T^{k_j}}
            g(r,s_{(j)},t_{(j)}) {s_{(j)}}^{\alpha_{(j)} + \beta_{(j)} + 1_{k_j}}
            {t_{(j)}}^{\alpha_{(j)} - \beta_{(j)}} \\
        \times
        (1 - |r|^2)^\lambda \prod_{l=1}^m r_l^{2k_l+2\kappa_l-1} \dif r_l
        \dif s_{(j)} \dif t_{(j)},
    \end{multline*}
    for every $\alpha, \beta \in \N^n$ such that $|\alpha_{(l)}| = |\beta_{(l)}| = \kappa_l$, for all $l = 1, \dots, m$, that satisfy $\alpha_{\widehat{(j)}} = \beta_{\widehat{(j)}}$.
\end{theorem}

\begin{corollary}\label{cor:Tkappaj_withg_forToeplitz}
    With the notation of Theorem~\ref{thm:Tkappaj_withf_forToeplitz} and Remark~\ref{rmk:Tblockdiagonalconst}, if $a \in L^\infty(\B^n)$ is a $(\kb,j)$-quasi-radial quasi-homogeneous symbol, then, for every $\kapb \in \N^m$, the coefficients of the linear map $(T_a^{(\lambda)})_{\kapb,j}$ are given by
    \begin{multline*}
        \langle T_a^{(\lambda)}(e_\alpha), e_\beta \rangle_\lambda =
        \frac{2^{m-1} \Gamma(n+\lambda+|\kapb|+1)}{\pi^{k_j} \Gamma(\lambda+1) \sqrt{\alpha_{(j)}! \beta_{(j)}!} \prod_{l\not=j} (k_l+\kappa_l-1)!} \times \\
        \times
        \int_{\tau(\B^m)\times S_+^{k_j-1} \times \T^{k_j}}
            g(r,s_{(j)},t_{(j)}) {s_{(j)}}^{\alpha_{(j)} + \beta_{(j)} + 1_{k_j}}
            {t_{(j)}}^{\alpha_{(j)} - \beta_{(j)}} \\
        \times
        (1 - |r|^2)^\lambda \prod_{l=1}^m r_l^{2k_l+2\kappa_l-1} \dif r_l
        \dif s_{(j)} \dif t_{(j)},
    \end{multline*}
    for every $\alpha, \beta \in \N^n$ such that $|\alpha_{(l)}| = |\beta_{(l)}| = \kappa_l$, for all $l = 1, \dots, m$, that satisfy $\alpha_{\widehat{(j)}} = \beta_{\widehat{(j)}}$.
\end{corollary}

\begin{remark}\label{rmk:Tblockdiagonaldescription}
    Let $a \in L^\infty(\B^n)$ be a $(\kb,j)$-quasi-radial quasi-homogeneous symbol, and, for a given $\lambda > -1$, consider the Toeplitz operator $T_a^{(\lambda)}$. Following the notation from Remark~\ref{rmk:Tblockdiagonalconst} and recollecting the information from the previous results we can provide the following description of the Toeplitz operator $T_a^{(\lambda)}$. In the first place, there is a block diagonal decomposition
    \begin{equation}\label{eq:block_lambdaquasi}
        T_a^{(\lambda)} = \bigoplus_{\kapb \in \N^m}
            T_a^{(\lambda)} |_{\cP_\kapb(\C^n)},
    \end{equation}
    obtained from Proposition~\ref{prop:EndU(kb,j,T)-tensorproduct} and Corollary~\ref{cor:HinvariantToeplitz}. For every $\kapb \in \N^m$, let us denote $\kapb_{\widehat{j}} = (\kappa_1, \dots, \kappa_{j-1}, \kappa_{j+1}, \dots, \kappa_m)$.
    Then, for every $\kapb \in \N^m$, the linear transformation given by the block $T_a^{(\lambda)} |_{\cP_\kapb(\C^n)}$ satisfies the following unitary equivalence
    \begin{equation}\label{eq:Mlambdaa(kapb)decomp}
        T_a^{(\lambda)} |_{\cP_\kapb(\C^n)} \asymp
            \underbrace{M_{\lambda,a}(\kapb) \oplus
            \dots \oplus M_{\lambda,a}(\kapb)}_{\dim \cP_{\kapb_{\widehat{j}}}(\C^{n-k_j})\;\; \text{times}},
    \end{equation}
    where the right-hand side is the direct sum of $\dim \cP_{\kapb_{\widehat{j}}}(\C^{n-k_j})$ identical copies of the matrix $M_{\lambda,a}(\kapb) \in M_{\dim \cP_{\kappa_j}(\C^{k_j}) \times \dim \cP_{\kappa_j}(\C^{k_j})}(\C)$, whose coefficients are given by the formulas from Corollaries~\ref{cor:Tkappaj_withf_forToeplitz} and \ref{cor:Tkappaj_withg_forToeplitz}. More precisely, the coefficients of the matrix $M_{\lambda,a}(\kapb)$ are given by
    \begin{multline}\label{eq:Mlambdaa(kapb)coeff-f}
        M_{\lambda,a}(\kapb)_{\alpha,\beta} =  \frac{2^{m-1} \Gamma(n+\lambda+|\kapb|+1)}{\pi^{k_j} \Gamma(\lambda+1) \sqrt{\alpha! \beta!} \prod_{l\not=j} (k_l+\kappa_l-1)!} \times \\
        \times
        \int_{\tau(\B^m)\times \mbS^{k_j}}
            f(r,\xi_{(j)}) {\xi_{(j)}}^{\alpha} {\overline{\xi}_{(j)}}^{\beta} (1 - |r|^2)^\lambda \prod_{l=1}^m r_l^{2k_l+2\kappa_l-1} \dif r_l \dif \xi_{(j)}
    \end{multline}
    when the symbol $a$ is expressed in terms of a function $f$ as in Lemma~\ref{lem:kbj-quasi}(1), and they are given by
    \begin{multline}\label{eq:Mlambdaa(kapb)coeff-g}
        M_{\lambda,a}(\kapb)_{\alpha,\beta} = \frac{2^{m-1} \Gamma(n+\lambda+|\kapb|+1)}{\pi^{k_j} \Gamma(\lambda+1) \sqrt{\alpha! \beta!} \prod_{l\not=j} (k_l+\kappa_l-1)!} \times \\
        \times
        \int_{\tau(\B^m)\times S_+^{k_j-1} \times \T^{k_j}}
            g(r,s_{(j)},t_{(j)}) {s_{(j)}}^{\alpha + \beta + 1_{k_j}}
            {t_{(j)}}^{\alpha - \beta} \\
        \times
        (1 - |r|^2)^\lambda \prod_{l=1}^m r_l^{2k_l+2\kappa_l-1} \dif r_l
        \dif s_{(j)} \dif t_{(j)},
    \end{multline}
    when $a$ is expressed in terms of a function $g$ as in Lemma~\ref{lem:kbj-quasi}(2). In both cases, $\alpha, \beta$ run through all multi-indices in $\N^{k_j}$ that satisfy $|\alpha| = |\beta| = \kappa_j$. The matrices $M_{\lambda,a}(\kapb)$ appearing with multiplicity $\dim \cP_{\kapb_{\widehat{j}}}(\C^{n-k_j})$, for every $\kapb \in \N^m$, yield all the information associated to the Toeplitz operator $T_a^{(\lambda)}$. Such information is subsumed in equations \eqref{eq:block_lambdaquasi}, \eqref{eq:Mlambdaa(kapb)decomp}, \eqref{eq:Mlambdaa(kapb)coeff-f} and \eqref{eq:Mlambdaa(kapb)coeff-g}.
\end{remark}

\subsection*{Acknowledgements}
This research was partially supported by SNI-Conacyt and Conacyt Grants 280732 and 61517.

\end{document}